\documentclass[11pt,reqno]{amsart}
\usepackage{geometry}                
\usepackage{graphicx}
\usepackage{amssymb}
\usepackage{epstopdf}
\usepackage{amsmath}
\usepackage{color}
\usepackage{hyperref}
\usepackage{pdfsync}
\usepackage{float}
\usepackage{verbatim}
\usepackage{tikz}
\usetikzlibrary{arrows,decorations.pathmorphing,backgrounds,positioning,fit,petri}
\usetikzlibrary{arrows,patterns,matrix}
\usetikzlibrary{decorations.markings}

\usepackage[all]{xy}
\xyoption{matrix}
\xyoption{arrow}
\usepackage[table]{colortbl}
 
\tikzset{help lines/.style={step=#1cm,very thin, color=gray},
help lines/.default=.5} 
\tikzset{thick grid/.style={step=#1cm,thick, color=gray},
thick grid/.default=1} 

\usepackage{tcolorbox}  
\tcbset{colframe=black,colback=white, width=(\linewidth-2mm), before=,after=\hfill,} 
\usepackage{graphicx}
\usepackage{tikz}
\usetikzlibrary{mindmap,shapes,snakes,matrix}

\tikzstyle{ann}=[fill=white, inner sep=1pt, font=\footnotesize{#1}]
\tikzstyle{annfar}=[inner sep=2pt, font=\footnotesize{#1}]
\tikzstyle{annfarer}=[inner sep=3pt, font=\footnotesize{#1}]
\tikzstyle{annrot}=[fill=white, text=blue!75!black, inner sep=1pt, font=\footnotesize{#1}]
\tikzstyle{wall}=[thick]
\tikzstyle{nullwall}=[thick, dotted]

\newtheorem{thm}{Theorem}[subsection]
\newtheorem{thm*}{Theorem}
\newtheorem{lem}[thm]{Lemma}
\newtheorem{cor}[thm]{Corollary}
\newtheorem{prop}[thm]{Proposition}
\newtheorem{conj}[thm]{Conjecture}

\theoremstyle{definition}
\newtheorem{defn}[thm]{Definition}

\newtheorem{exm}[thm]{Example}

\newtheorem{rmk}[thm]{Remark}

 \newcounter{tmp}

\numberwithin{equation}{section}

\DeclareMathOperator{\Hom}{Hom}%
\DeclareMathOperator{\Ext}{Ext}%
\DeclareMathOperator{\End}{End}%
\DeclareMathOperator{\add}{add} 
\DeclareMathOperator{\modd}{mod}%
\newcommand{\commentout}[1]{}

\newcommand{\Top}{\operatorname{top}}
\newcommand{\rad}{\operatorname{rad}}
\newcommand{\soc}{\operatorname{soc}}
\newcommand{\Ker}{\operatorname{Ker}}
\newcommand{\cok}{\operatorname{Coker}}
\newcommand{\Ima}{\operatorname{Im}}

\begin{document}

\title{Functors and Morphisms determined by subcategories}
\subjclass[2010]{18A05,16G20, 16G70}

\keywords{morphisms determined by objects, Krull-Schmidt categories, almost split sequences, strongly locally finite quivers}

\author{Shijie Zhu}
\address{Department of Mathematics\\Northeastern University\\Boston, MA 02115}
\email{zhu.shi@husky.neu.edu} 


\maketitle

\begin{abstract}
  We study the existence and uniqueness of minimal right determiners in various categories. Particularly in a $\Hom$-finite hereditary abelian category with enough projectives, we prove that the Auslander-Reiten-Smal{\o}-Ringel formula of the minimal right determiner still holds.  As an application, we give a formula of minimal right determiners in the category of finitely presented representations of strongly locally finite quivers.
\end{abstract}
\tableofcontents

\section{Introduction}

The concept of morphisms determined by objects was introduced by Maurice Auslander in his Philadelphia Notes \cite{A} published in 1978. Although it was originally introduced in a functorial way, it was mainly applied to studying categories of finitely generated modules over Artin algebras.  Throughout thirty years, mathematicians, including Henning Krause \cite{K1}\cite{K4}, Claus Michael Ringel \cite{R1}\cite{R2}, Xiao-Wu Chen and Jue Le \cite{C}\cite{CL}, Jan \v{S}aroch \cite{S} and many others, have been further developing Auslander's idea of morphisms determined by objects. Krause successfully develops this theory in triangulated categories \cite{K1}. Ringel also gives a clear outline of Auslander's ideas without using functorial language \cite{R1}. 

 The general existence theorems of right (left) determiners in module categories were established in \cite{A}. Although in the category of finitely generated modules over Artin algebra, the minimal right (left) determiner of a morphism does exist and is unique, their existence and uniqueness have not been studied in general.  Our motivation was inspired by Question 5.3 in \cite{K1}, where the author raises the question about the existence of a minimal class of objects which can determine a morphism in some category. (See Definition \ref{mdo} and \ref{mrd defn} for definitions of (minimal) right determiners.)

  We show that for a given morphism in a Krull-Schmidt category, when the socle of the cokernel functor is essential, there exists a unique minimal right determiner (Theorem \ref{exist}).
  
  While these two properties are equivalent in the module categories of artin algebras, the existence of the essential socle functor only implies the existence of minimal right determiner in general. As shown in Example \ref{non-strong ex}, the existence of minimal right determiner does not imply the existence of essential socle functor. However, in general for arbitary Krull-Schmidt categories  the socle functor $\soc\cok(-,f)$ being essential is equivalent only to the existence of strong minimal right determiner (See Definition \ref{strong defn}).
  
%
  \begingroup
\setcounter{tmp}{\value{thm}}
\setcounter{thm}{0} 
\renewcommand\thethm{\Alph{thm}}

  \begin{thm}\label{exist}
Let $\mathcal C$ be a Krull-Schmidt category. Let $f:X\rightarrow Y$ be a morphism and let $\Hom(-,X)\xrightarrow{(-,f)}\Hom(-,Y)\stackrel{\pi}\longrightarrow\cok(-,f)\rightarrow 0$ be the induced exact sequence of functors.  Then $\soc \cok(-,f)$ is essential if and only if $f$ has a strong minimal right determiner.
\end{thm}
  
   We also show that if a morphism has a minimal right determiner then it is unique (Theorem \ref{unique}) in certain classes of Krull-Schmidt categories (see Remark \ref{D} for classes of such categories). 
   
   \begin{thm}
Let $\mathcal C$ be a Krull-Schmidt category and let $\mathcal D$ be a minimal right determiner of  $f:X\rightarrow Y$. Suppose for each $Z\in \mathcal D$ there is a natural number $n$ such that $(\rad\End(Z))^n\Hom(Z,Y)\subseteq \Ima\Hom(Z,f)$. Then $\mathcal D$ is the unique minimal right determiner.
\end{thm}

More importantly, the existence of the minimal right determiner is closely related with the existence of  almost split sequences. We also give a formula of the minimal right determiner of a morphism in a Krull-Schmidt $\Hom$-finite hereditary abelian category:  

\begin{thm} \label{C}
Let $\mathcal C$ be a Krull-Schmidt $\Hom$-finite hereditary abelian $k$-category. Let $f$ be a morphism in $\mathcal C$. Then $f$ has a strong minimal right determiner if and only if: 
\begin{enumerate}
\item 
each indecomposable summand of the intrinsic kernel  $\widetilde\Ker(f)$ is the starting term of an almost split seqeunce; 
\item the socle $\soc\cok(f)$ is essential;
\item each simple subobject of $\cok f$ has a projective cover.
\end{enumerate}
Furthermore, if the minimal right determiner exists, it is of the same form as in the Auslander-Reiten-Smal{\o}-Ringel formula, i.e. $\mathcal D=\add( \tau^-\widetilde\Ker f\oplus P(\soc\cok f))$. Here $\tau^-$ denotes the right end terms of the almost split sequences. 
\end{thm}

As an application in Section \ref{inf q} we will concentrate on the strongly locally finite quivers (see Definition \ref{slf}) introduced by Raymundo Bautista, Shiping Liu and Charles Paquette \cite{BLP}. Suppose $Q$ is a strongly locally finite quiver. The category of finitely presented representations $rep^+(Q)$, as defined in \ref{rep(q)}, is a Krull-Schmidt, $\Hom$-finite, $\Ext$-finite, hereditary and abelian category (\cite{BLP}, 1.15). 
Recall that a path  of the following form: $$\bullet\rightarrow\bullet\rightarrow\bullet\rightarrow \cdots$$ is called a right infinite path.
We show that for a strongly locally finite quiver $Q$ having no right infinite path, every morphism in $rep^+(Q)$ admits a minimal right determiner:

 \begin{thm} \label{no r inf path}
Let $Q$ be a strongly locally finite quiver. Then the following are equivalent:\\
$(1)$ $Q$ has no right infinite paths.\\
$(2)$ $rep^+(Q)$ is left Auslander-Reiten (indecomposable objects have left almost split morphisms). \\
$(3)$ Any morphism $f:X\rightarrow Y$ in $rep^+(Q)$ has a unique minimal right determiner.\\
\end{thm}

\endgroup

{\bf Acknowledgement:}  This project is under the guidance of Professor Gordana Todorov. The author will thank Charles Paquette and Shiping Liu for helpful discussions at the XXVII$^{th}$ Meeting on Representation Theory of Algebras; Dan Zacharia for his hospitality during my visit at Syracuse University; as well as Yingdan Ji for pointing out mistakes during the seminars.

\section{Morphisms determined by objects and subcategories}
Throughout the paper, we always assume our category $\mathcal C$ is skeletally small (i.e. isomorphism classes of objects form a set). 

\subsection{Definitions}  Let $\mathcal C$ be an additive category. We will generalize the notion of morphisms determined by objects to morphisms determined by subcategories in $\mathcal D\subseteq \mathcal C$. Throughout the paper subcategories are considered to be full subcategories closed under direct sums and summands. 

\begin{defn}\label{mdo}
Suppose $\mathcal D$ is a full subcategory of $\mathcal C$ closed under direct sums and summands. A morphism $f:X\rightarrow Y$ is said to be {\bf right determined by $\mathcal D$} (or {\bf right $\mathcal D$-determined}) if for every morphism $f^\prime: X^\prime\rightarrow Y$ the following conditions are equivalent:

\noindent $(a)$ The morphism $f^\prime$ factors through $f$.\\
\noindent $(b)$ For any object $Z\in \mathcal D$ and any morphism $g: Z\rightarrow X^\prime$ the composition $f^\prime g$ factors through~$f$. 
\vskip5pt
If $f$ is right determined by $\mathcal D$, $\mathcal D$ is also called {\bf a right determiner of $f$}. If $\mathcal D$ contains only finitely many non-isomorphic indecomposable objects, then we say $f$ is {\bf right determined by an object}.
\end{defn}

Notice that the condition $(a)$ always implies condition $(b)$. So the definition can also be equivalently phrased in the following way which will be used often:

\begin{prop}\label{contrapositive}
A morphism $f:X\rightarrow Y$ is right $\mathcal D$-determined if and only if for each morphism $f^\prime:X^\prime \rightarrow Y$ which cannot factor through $f$ there exists an object $Z\in \mathcal D$ and a morphism $g:Z\rightarrow X^\prime$ such that $f^\prime g$ cannot factor through $f$.
\end{prop}


\begin{rmk} Here we state some facts which are useful and not hard to check.
\begin{enumerate}
\item Although $\mathcal D$ is defined as a full subcategory, we are only interested in its objects. 
\item Every subcategory of $\mathcal C$ contains the zero object. Denote by $\mathbf 0$ the subcategory containing only the zero object in $\mathcal C$. Then a morphism is a split epimorphism if and only if it is right $\mathbf 0$-determined.
\item Suppose $\{D_i\}_{\i\in I}$ is a complete set of representatives of isomorphism classes of indecomposable objects in $\mathcal D$, then we say $\mathcal D$ is additively generated by $\{D_i\}$ (or $D_i$'s are additive generators of $\mathcal D$). Denote $\mathcal D=\add\{D_i\}$. When $\{D_i\}_{i\in I}$ is a finite set, we say that $f$ is right determined by the object $\oplus_{i\in I}D_i$.

\item If $\mathcal C=\Lambda\textendash\modd$ for an artin algebra $\Lambda$, then any morphism is right determined by some $\mathcal D$ containing finitely many non-isomorphic indecomposable objects. i.e. every morphism is right determined by an object. (\cite{ARS}[X, Prop. 2.4])
\item A morphism $f$ is left $\mathcal D$-determined, if $f$ is right $\mathcal D$-determined in the opposite category.
\end{enumerate}
\end{rmk}

\vskip10pt

The following Lemma can be verified by the universal property of direct sums.  

\begin{lem} \label{sum}
Let $f:X\rightarrow Y$ be a morphism in an additive category $\mathcal C$. Let $Z=\mathop{\bigoplus}\limits_{i\in I}Z_i$ and $h=(h_i): Z\rightarrow Y$ be a morphism.Then the following statements are equivalent:
\begin{enumerate}
\item The morphism $h:Z\rightarrow Y$ can factor through $f$.
\item The morphisms $h_i:Z_i\rightarrow Y$  can factor through $f$ for $\forall i\in I$.
\end{enumerate}
\end{lem}

%

%
%

Applying this Lemma, Definition \ref{mdo} and Proposition \ref{contrapositive} can be rephrased in terms of indecomposable objects  in the following way:
\begin{prop}\label{mdo2}
Suppose $\mathcal D$ is a full subcategory of $\mathcal C$ closed under direct sums and summands. A morphism $f:X\rightarrow Y$ is right determined by $\mathcal D$, if and only if for every morphism $f^\prime: X^\prime\rightarrow Y$ the following conditions are equivalent:
\begin{enumerate}
\item The morphism $f^\prime$ factors through $f$.
\item For any indecomposable object $Z\in \mathcal D$ and any morphism $g: Z\rightarrow X^\prime$ the composition $f^\prime g$ factors through $f$.
\end{enumerate}
\end{prop}

 \begin{cor}\label{contrapositive2}
A morphism $f:X\rightarrow Y$ is right $\mathcal D$-determined if and only if for each morphism $f^\prime:X^\prime \rightarrow Y$ which cannot factor through $f$ there exists an indecomposable object $Z\in \mathcal{D}$ and a morphism $g:Z\rightarrow X^\prime$ such that $f^\prime g$ cannot factor through $f$.
\end{cor}

\begin{defn}\label{r minimal}
 A morphism $f:X\rightarrow Y$ is called {\bf right minimal} if for any $g:X\rightarrow X$ satisfying $fg=f$, the morphism $g$ is an automorphism.
 \end{defn}
It is easy to check that the following are equivalent:\\
(1) $f:X\rightarrow Y$ is right minimal.\\
(2) A direct summand $X^\prime$ of $X$ equals zero whenever $f(X^\prime)=0$.

 For each morphism $f:X\rightarrow Y$, we can always assume $f=(f_1,0): X_1\oplus X_2\rightarrow Y$ such that $f_1:X_1\rightarrow Y$ is right minimal. The morphisms $f_1$ is called the {\bf right minimal version} $f_1$. 
 
 Let $f$ be a morphism in an abelian category $\mathcal C$ with right minimal version $f_1$. Then the object $\widetilde{\Ker f}:=\Ker f_1$ is called the {\bf intrinsic kernel} of $f$.
 
 \vskip 5pt
The following Proposition is easy to check by definition:

\begin{prop} \label{r min ver}
A morphism $f$ is right determined by $\mathcal D$ if and only if its right minimal version $f_1$ is right determined by $\mathcal D$.  
 \end{prop}

\vskip10pt

\subsection{The minimal right determiner}\label{mrd}
By definition, if $\mathcal D$ is a right determiner of $f$ and $\mathcal{D}\subseteq \mathcal{D}^\prime$ then $\mathcal {D}^\prime$ is also a right determiner of $f$. So we are particularly interested in the ``minimal'' one under containment among all the right determiners.

\begin{defn}\label{mrd defn}
Let $f$ be a morphism in an additive category $\mathcal C$. A full subcategory $\mathcal D$ closed under direct sums and summands is said to be {\bf a minimal right determiner} of $f$, if $\mathcal D$ is a right determiner of $f$ and $f$ is not right determined by any full subcategory $\mathcal{D}^\prime$ closed under direct sums and summands such that $\mathcal {D}^\prime\subsetneq \mathcal{D}$.
\end{defn}

\begin{rmk}\label{mrd not minimal} Let $\mathcal C$ be an additive category and $f$ be a morphism in $\mathcal C$.
\begin{enumerate}
\item A morphism $f$ is a split epimorphism if and only if $\bf 0$ is the minimal right determiner.
\item A morphism $f$ has a minimal right determiner $\mathcal D$ if and only if its right minimal version $f_1$ has a minimal right determiner $\mathcal D$.
\item Any morphism has a right determiner $\mathcal{D}=\mathcal C$, but there exists morphisms which do not have a minimal right determiner.
See Example \ref{n_ex}.
\end{enumerate}
\end{rmk}


Auslander, Reiten, Smal{\o} and Ringel give an explicit formula of the minimal right determiner in the category of finitely generated modules over an artin algebra.

\begin{thm}\cite{ARS}\cite{R2}
Let $\Lambda$ be an artin algebra and  $\mathcal C=\Lambda\textendash\modd$ be the category of finitely generated modules over $\Lambda$. Let $f:X\rightarrow Y$ be a morphism in $\mathcal C$. Then:
\begin{enumerate}
\item The minimal right determiner is a  summand of $\tau^-\widetilde{\Ker f}\oplus P(\soc \cok f)$, where $\tau$ is the Auslander-Reiten translation, $\widetilde{\Ker f}$ is the intrinsic kernel of $f$ and $P(-)$ is the projective cover. 
\item Furthermore, if $\Lambda$ is hereditary, then $\tau^-\widetilde{\Ker f}\oplus P(\soc \cok f)$ is the minimal right determiner of $f$.
\end{enumerate}
\end{thm}

The following Lemma shows that every indecomposable object in the minimal right determiner $\mathcal D$ is necessary in checking which morphisms factor through $f$.

\begin{lem}\label{in mrd}
Let $\mathcal C$ be an additive category. Let $f:X\rightarrow Y$ be a morphism in $\mathcal C$ which has a minimal right determiner $\mathcal D$. Let $Z$ be a non-zero indecomposable object in $\mathcal D$. Then there exists a morphism $g:T\rightarrow Y$ which cannot factor through $f$ satisfying the following conditions:
\begin{enumerate}
\item There is a morphism $\psi:Z\rightarrow T$ such that $g\psi$ cannot factor through $f$.
\item For any indecomposable object $Z^\prime\in \mathcal{D}$, $Z'\not\simeq Z$ and any morphism $\varphi:Z^\prime\rightarrow T$, the composition $g\varphi$ can factor through $f$.
\end{enumerate}
\end{lem}
\begin{proof} Let $\mathcal A=\{g:T\rightarrow Y\mid\ g \text{\ cannot\  factor\  through\ } f \}$. Notice that $f$ is not an isomorphism (otherwise $\mathcal {D}={\bf 0}$). Therefore $\mathcal A$ is not empty.

First, suppose $\forall g\in\mathcal A$ condition (2) fails. Then consider the full subcategory $\mathcal{D}^\prime\subsetneq \mathcal{D}$ consisting of objects $Z^\prime\in \mathcal{D}$ which have no direct summands isomorphic to $Z$. By Corollary \ref{contrapositive2}, $f$ is right determined by $\mathcal{D}^\prime$, which is a contradiction to the minimality of $\mathcal D$.

Therefore there exists a morphism $g:T\rightarrow Y$ such that (2) holds. Then because of the fact that $\mathcal D$ is a right determiner of $f$, $g$ must satisfy condition (1). 
\end{proof}

%
%
We say a morphism $f$ has a {\bf unique minimal right determiner} if both $\mathcal D$ and $\mathcal{D}^\prime$ being minimal right determiners of $f$ implies $\mathcal {D}=\mathcal {D}^\prime$.

Notice that $\bf 0$ is always the unique minimal right determiner of a split epimorphism.

In the category of finitely generated modules over an artin algebra, it can be shown that the minimal right determiner of a morphism is always unique. 
However, it is not true in general. See Example \ref{1221}.      

\vskip10pt
Let $\mathcal S$ be the collection of all the right determiners of $f$. Then $\mathcal S$ is a poset with the partial order defined by the inclusion.
 A minimal element in $\mathcal S$ is a minimal right determiner of $f$. The minimum element in $\mathcal S$, if it exists, is the unique minimal right determiner of $f$.
\begin{lem}\label{poset intersection}
Let $\mathcal S$ be the poset of all the right determiners of $f$. If $\mathcal S$ is closed under arbitary intersections, then $f$ has a unique minimal right determiner. 
\end{lem}
\begin{proof} Because $\mathcal S$ is closed under intersections, $\mathop{\bigcap}\limits_{\mathcal D\in\mathcal S}\mathcal D$ will be the minimum element of $\mathcal S$.\end{proof}
Notice that the converse of Lemma \ref{poset intersection} is not true, see Example \ref{wedge close}.  
 

\section{A functorial approach}
In \cite{A}, the original concept of morphisms determined by objects was introduced in a functorial way. We will give a brief introduction as well as fix our notation in this section. One will find that instead of being abstract, it is the most efficient way of understanding the determined morphisms. Since most results are well-known, we will not give the proofs of them. A good reference is Auslander's original paper \cite{A}.

\subsection{Subfunctors determined by objects}
Recall that a category is {\bf preadditive} if $\Hom(X,Y)$ are abelian groups and the composition of morphisms $(X,Y)\times(Y,Z)\rightarrow (X,Z)$ given by $(f,g)\rightarrow gf$ is bilinear (we often use $(-,-)$ for $\Hom(-,-)$). An {\bf additive} category is a preadditive category admitting finite products and coproducts. If $\mathcal C$ and $\mathcal D$ are both preadditive, a covariant functor $F: \mathcal C\rightarrow \mathcal D$ is said to be {\bf additive} if the morphisms $F:(X,Y)\rightarrow (F(X),F(Y))$ are group homomorphisms. We assume our functors are all additive.

Denote $(\mathcal C,\mathcal D)$ to be the category of additive functors from category $\mathcal C$ to $\mathcal D$. In particular, if $\mathcal C$ is a preadditive category and $\mathcal D=Ab$ is the category of abelian groups, the category $(\mathcal C^{op}, Ab)$ of all the contravariant functors from $\mathcal C$ to $Ab$ is known to be an abelian category.

\vskip5pt
Let $\mathcal C$ be an additive category in this subsection.
The following Lemma is crucial to the construction of subfunctors.

\begin{lem}[\cite{A}, I 1.1]\label{func exist}
Let $G$ be a functor in $(\mathcal C^{op}, Ab)$. If for each $X$ in $\mathcal C$ we are given a subgroup $A_X$ of $G(X)$ such that $G(f)(A_Y)\subseteq A_X$ for all morphisms $f: X\rightarrow Y$ in $\mathcal C$, then there is a unique subfunctor $F$ of $G$ such that $F(X)=A_X$ for all $X$ in $\mathcal C$.
\end{lem}

\subsection{A construction of subfunctor $G_H$}\label{construction} Let $G$ be a functor in $(\mathcal C^{op}, Ab)$, $X$ be an object of $\mathcal C$ and $H$ be an $\End(X)^{op}$-submodule of $G(X)$. Using the preceding Lemma, we can construct a subfunctor $G_H$ of $G$ in the following way such that $G_H$ is the maximal subfunctor of $G$ satisfying $G_H(X)=H$.
  For each $C\in\mathcal C$ define $A_C$ to be the subgroup of $G(C)$ consisting of all $x$ in $G(C)$ such that for each $f:X\rightarrow C$ the element $G(f)(x)$ in $G(X)$ is contained in $H$. It is easy to check that for each morphism $g: U\rightarrow V$ in $\mathcal C$ we have $G(g)(A_V)\subseteq A_U$. Thus by Lemma \ref{func exist}, there is a subfunctor, denoted by $G_H$ such that $G_H(C)=A_C$ for all $C\in \mathcal C$. Furthermore, since $H$ is an $\End(X)^{op}$-submodule of $G(X)$, $G_H(X)=H$. 

\vskip10pt
The following definition of ``subfunctors determined by objects'' is due to Auslander \cite{A}.
\begin{defn}\label{functor det obj def}
A subfunctor  $F$ of a functor $G:\mathcal C^{op}\rightarrow Ab$ is said to be {\bf determined by} $X\in Obj(\mathcal C)$ if a subfunctor $F^\prime$ of $G$ is contained in $F$ whenever $F^\prime(X)\subseteq F(X)$.
\end{defn}

\begin{rmk} It is worth mentioning the following easy observations:
\begin{enumerate}
\item If $F$ is a subfuctor of $G$ determined by $X$, then  $F(X)=G(X)$ implies $F=G$. 
\item A subfunctor $F$ of $G$ is determined by $X$ if and only if $F=G_{F(X)}$.
\end{enumerate}
\end{rmk}

The notion of ``subfunctors determined by objects'' can be easily generalized to ``subfunctors determined by subcategories'':

\begin{defn}
A subfunctor  $F$ of a functor $G:\mathcal C^{op}\rightarrow Ab$ is said to be {\bf determined by} subcategory $\mathcal D$ if a subfunctor $F^\prime$ of $G$ is contained in $F$ whenever $F^\prime(Z)\subseteq F(Z)$ for all $Z\in\mathcal D$.
\end{defn}

The notion of ``subfunctors determined by objects'' was originally used as the definition of morphisms determined by objects by Auslander \cite{A}. For the convenience of the readers, we show the equivalence of these two definitions in the generalized version for subcategories:
\begin{prop}\label{det fun}
 Let $f:B\rightarrow C$ be a morphism in $\mathcal C$ and $\mathcal D$ a full subcategory of $\mathcal C$. The morphism $f$ is right $\mathcal D$-determined if and only if the subfunctor $\Ima(-,f)$ of $(-,C)$ is determined by the subcategory $\mathcal D$.
\end{prop}

\begin{proof}

``if part'': Suppose a given morphism $f^\prime: A\rightarrow C$ satisfies: for any $Z\in \mathcal D$ and morphism $g: Z\rightarrow A$ the composition $f^\prime g$ factors through $f$. So $\Ima(Z,f^\prime)\subseteq\Ima(Z,f)$ for all $Z\in\mathcal D$. By our assumption, it follows that $\Ima(-,f^\prime)\subseteq\Ima(-,f)$. Therefore, we have the following commutative diagram.
$$
\xymatrix{(-,A)\ar@{->>}[r]\ar@{-->}[d]^\alpha &\Ima(-,f^\prime)\ar@{^(->}[d]\ar@{^(->}[r]&(-,C)\ar@{=}[d]\\
          (-,B)\ar@{->>}[r]&\Ima(-,f)\ar@{^(->}[r]&(-,C)}
$$
By Yoneda lemma, $\alpha$ is given by $(-,h)$ for some $h:A\rightarrow B$, and hence $f^\prime=fh$.

``only if part'': Suppose $G$ is a subfunctor of $(-,C)$ and $G(Z)\subseteq \Ima(Z,f)$ for all $Z\in\mathcal D$.  We want to show that $G\subseteq \Ima(-,f)$.  We just need to prove that $G(B')\subseteq \Ima(B',f)$ for all $B'\in\mathcal C$.
Let $f^\prime\in G(B')\subseteq\Hom(B',C)$. Since for any $Z\in\mathcal D$ and any morphism $g:Z\rightarrow B'$, the composition $f^\prime g\in G(Z)\subseteq\Ima(Z,f)$. So $f'g$ factors through $f$. Since $f$ is right determined by $\mathcal D$, it follows that $f^\prime$ factors through $f$. So $f^\prime\in \Ima(B',f)$ and hence $G\subseteq \Ima(-f)$.
\end{proof}

\subsection{Simple functors}

In this subsection, let $\mathcal C$ be an additive category. We are going to recall some basics about simple functors. Recall that a functor $F$ in $(\mathcal C^{op}, Ab)$ is said to be a {\bf simple functor} if a) $F\neq 0$ and b) 0 and $F$ are the only subfunctors of $F$. A functor $F$ is called {\bf semisimple} if $F$ is a direct sum of simple functors.

Let $\{G_i\}_{i\in I}$ be a family of subfunctors of $G$ in $(\mathcal C^{op}, Ab)$. For each $C$ in $\mathcal C$ define $A_C=\mathop{\bigcap}\limits_{i\in I}G_i(C)$ which is a subgroup of $G(C)$. Since each $G_i$ is a functor, it follows that for each morphism $f:X\rightarrow Y$ we have $G(f)(A_Y)\subseteq A_X$. Thus, by Lemma \ref{func exist} there is a unique subfunctor $F$ of $G$ such that $F(C)=A_C$ for all $C\in\mathcal C$. Define $\mathop{\bigcap}\limits_{i\in I}G_i:=F$ and call it the {\bf intersection} of the $G_i$.
\vskip5pt

A subfuctor $F$ of $G$ in $(\mathcal C^{op}, Ab)$ is called {\bf essential} if for any non-zero subfunctor $H\subseteq G$, it follows that $F\cap H\neq0$.

Equivalently, $F$ is an essential subfunctor of $G$ if and only if for any non-zero subfunctor $H \subseteq G$ there is a non-zero subfunctor $F^\prime\subseteq F$ such that $F^\prime$ is also a subfunctor of $H$.
\vskip5pt

The {\bf socle} of $G$ is the maximal semisimple subfunctor of $G$ in $(C^{op}, Ab)$, denoted by $\soc G$. By convention, $\soc G = 0$ if $G$ has no simple subfunctor. We say $\soc G$ is essential if it is an essential subfunctor of $G$. Then we have the following easy observation.

\begin{prop}\label{ess equiv}
The subfunctor $\soc G$ of $G$ is essential if and only if every non-zero subfunctor $H\subset G$ has a simple subfunctor.
\end{prop}

\begin{proof} Suppose $\soc G$ is essential, then for any non-zero subfunctor $H$, there is a subfunctor $F^\prime\subseteq \soc G$ such that, $F^\prime\subseteq H$. Since $F^\prime \subseteq\soc G$  is semisimple, there is a simple subfunctor $S\subseteq F^\prime$. Hence $S\subseteq H$.

Conversely, let $H\subseteq G$ be a non-zero subfunctor and assume $H$ has a simple subfunctor $S$. Since any simple subfunctor $S$ is a direct summand of $\soc G$, this implies that $S\subseteq H\cap \soc G$. So $\soc G$ is essential.   \end{proof} 

The following Lemma says that every simple functor is finitely generated.
\begin{lem}[\cite{A}, II 1.7]\label{simple fun}
Let $S$ be a simple functor in $(\mathcal C^{op}, Ab)$. Then $S(C)\neq 0$ for some $C\in \mathcal C$ and $S\simeq (-,C)/(-,C)_H$ where $H$ is a maximal right ideal of $\End(C)$.
\end{lem}

The {\bf radical} of $G$ is the intersection of all the maximal subfunctors of $G$, denoted by $rG$ or $\rad G$.  In particular in an additive category, $\rad_{\mathcal C}(X,Y):=r(-,Y)(X)=\{h\in(X,Y) | 1_X-gh$ is invertible for any $g\in(Y,X)\}$.
 The following Lemma describes when simple functors have projective covers.
 
\begin{lem} [\cite{A}, II 1.9]
Let $\mathcal C$ be an additive category. Let $S$ be a simple functor in $(\mathcal C^{op},Ab)$ and $\alpha:(-,C)\rightarrow S$ a nonzero morphism. Then the following are equivalent.\\
$(a)$ $\alpha: (-,C)\rightarrow S$ is a projective cover.\\
$(b)$ $\End(C)$ is local.\\
$(c)$ $\Ker \alpha=r(-,C)$.
\end{lem}

\begin{cor}[\cite{A}, II 1.11] \label{simple fun 2}
Let $\mathcal C$ be an additive category. There is a bijection between

$\{$the isomorphism classes of objects in $\mathcal C$  with local endomorphism rings$\}\stackrel{\psi}\rightarrow $\\ \indent$\{$the isomorphism classes of simple functors in $(\mathcal C^{op}, Ab)$ having projective covers$\}$,

where $\psi(C)=(-,C)/r(-,C)$ for each $C\in\mathcal C$ with $\End(C)$ local.
 \end{cor}

\vskip10pt

\subsection{Krull-Schmidt categories} Recall that an additive category $\mathcal C$ is called {\bf Krull-Schmidt} if any object $M\in\mathcal C$ has a unique (up to permutation) decomposition: $M=\mathop{\oplus}\limits_{i=1}^{n}M_i$ where $\End_{\mathcal C}(M_i)$ is a local ring for all $1\leq i\leq n$.

\vskip10pt
We will summarize some well-known facts about the Krull-Schmidt categories. 

An additive category has {\bf split idempotents}  if every idempotent endomorphism $\phi=\phi^2$ of an object $X$ splits, that is, there exists a factorization $X\stackrel{\iota} \rightarrow Y\stackrel{\pi}\rightarrow X$ of $\phi$ with  $\iota\pi=id_Y$. In particular, abelian categories have split idempotents.

A ring $R$ is called {\bf semi-perfect} if every finitely generated left (right) $R$-module has a projective cover. For example, local rings, left (right) artinian rings and finite dimensional $k$-algebras are semi-perfect.

\begin{prop} [\cite{K3}, 4.3]
An additive category is Krull-Schmidt if and only if it has split idempotents and the endomorphism ring of every object is semi-perfect.
\end{prop}

Let $k$ be a field. Recall that an additive $k$-category is {\bf Hom-finite} if the $\Hom$ spaces are of finite $k$-dimension.  Consequently, a $\Hom$-finite additive $k$-category is Krull-Schmidt if and only if it has split idempotents. In particular, a $\Hom$-finite abelian $k$-category is always Krull-Schmidt.

\begin{prop}[\cite{K2}, A.1]\label{ks proj cov}
For an essentially small (i.e. a category equivalent to a small category) additive category $\mathcal C$ with split idempotents the following are equivalent:\\
$(1)$ $\mathcal C$ is a Krull-Schmidt category.\\
$(2)$ Every finitely generated contravariant functor admits a projective cover.
\end{prop}

\begin{cor}\label{1-1}  If $\mathcal C$ is a Krull-Schmidt additive category, then each simple functor in $(\mathcal C^{op}, Ab)$ has a projective cover. Hence, there is a bijection between  indecomposable objects in $\mathcal C$ and all simple functors.
\end{cor}
\begin{proof}
Let $S$ be a simple functor. By Lemma \ref{simple fun}, $S$ is finitely generated and hence by Proposition \ref{ks proj cov}, it has a projective cover.
\end{proof}

Recall that a functor $F$ in an additive category $(\mathcal C^{op}, Ab)$ is called {\bf finitely generated} if there is an epimorphism: $(-,C)\rightarrow F$ for some $C\in\mathcal C$. A functor $F$ is called {\bf finitely presented} if there is an exact sequence $(-,B)\rightarrow (-,C)\rightarrow F\rightarrow 0$ for some $B,C\in\mathcal C$. It also follows that if $\mathcal C$ is a Krull-Schmidt additive category, then every finitely presented functor $F$ in $(\mathcal C^{op}, Ab)$ has a minimal projective presentation:
$$
(-,M)\rightarrow(-,N)\rightarrow F\rightarrow 0.
$$

If $\mathcal C$ is an abelian category, finitely presented functors have this well-known locally coherent property:

\begin{lem}\label{coh}
Let $\mathcal C$ be an abelian category and suppose $F$ in $(\mathcal C^{op}, Ab)$ is finitely presented. If $G$ is a finitely generated subfunctor of $F$, then $G$ is finitely presented.
\end{lem}

It is well-known that in abelian category, an epimorphism $P\stackrel{f}\rightarrow X$ is a projective cover if and only if $P$ is projective and $f$ is right minimal (See Definition \ref{r minimal}; \cite{K3}, 3.4 ). Notice that it is equivalent to say that $f$ is an essential epimorphism (See \cite{ARS} I.4.1). So we have the following result about the existence of projective cover in Krull-Schmidt abelian categories.
 
\begin{prop}\label{s cover}
Suppose $\mathcal C$ is a Krull-Schmidt abelian category. If there is an epimorphism $\pi:P\rightarrow X$ in $\mathcal C$ where $P$ is projective, then $X$ has a projective cover.
\end{prop}
\begin{proof} Since $\mathcal C$ is Krull-Schmidt, $P=\oplus_{i=1}^n P_i$, where $P_i$'s are indecomposable summands of $P$ with local endomorphism ring. Take $\tilde\phi:\tilde P\rightarrow X$ to be the right minimal version of $\phi$. Since $\Ima\phi=\Ima \tilde\phi=X$, $\tilde\phi$ is a right minimal epimorphism. Hence $P$ is the projective cover of $X$. \end{proof}

\subsection{Almost split sequences}\label{ars def}

We recall some facts about almost split sequences in additive categories. It could be regarded as a generalization of the classical Auslander-Reiten theory in module categories.

Let $\mathcal C$ be an arbitrary category and let $g:X\rightarrow Y$ be a morphism in $\mathcal C$. Then\\
(1) $g$ is said to be {\bf right minimal} if for any $h:X\rightarrow X$ satisfying $gh=g$, $h$ is an automorphism.\\
(2) $g$ is said to be {\bf right almost split} if (a) $g$ is not a split epimorphism. (b) If $f: Z\rightarrow Y$ is not a split epimorphism, there is a morphism $h: Z\rightarrow Y$ such that $gh=f$.\\
(3) $g$ is said to be {\bf minimal right almost split} if $g$ is both right minimal and right almost split.
\vskip5pt

For the convenience of the readers, we also give the dual definitions:\\
(1') $g$ is said to be {\bf left minimal} if for any $h:Y\rightarrow Y$ satisfying $hg=g$, $h$ is an automorphism.\\
(2') $g$ is said to be {\bf left almost split} if (a) $g$ is not a split monomorphism. (b) If $f: X\rightarrow Z$ is not a split monomorphism, there is a morphism $h: Y\rightarrow Z$ such that $hg=f$.\\
(3') $g$ is said to be {\bf minimal left almost split} if $g$ is both left minimal and left almost split.

The notion of right (left) almost split morphisms can be interpreted in terms of functors as the follows: 

\begin{prop} [\cite{A}, II 2.3]
Let $\mathcal C$ be an additive category and $f:B\rightarrow C$ be a morphism in $\mathcal C$. Then the following are equivalent:
\begin{enumerate}
\item $\Ima(-,f)$ is a maximal subfunctor of $(-,C)$ and $\End(C)$ is local.
\item $f$ is right almost split.
\end{enumerate}
\end{prop}

\begin{prop} [\cite{A}, II 2.7]
Let $\mathcal C$ be an additive category. Suppose $f:B\rightarrow C$ is a morphism in $\mathcal C$ such that $(-,B)\stackrel{(-,f)}\rightarrow (-,C)\rightarrow \cok(-,f)\rightarrow 0$ is exact and $\cok(-,f)$ is a simple functor. Then the following are equivalent:
\begin{enumerate}
\item $(-,B)\stackrel{(-,f)}\rightarrow (-,C)\rightarrow \cok(-,f)\rightarrow 0$ is a minimal projective presentation of $\cok(-,f)$.
\item $f$ is minimal right almost split.
\end{enumerate}
\end{prop}

We omit the dual versions of this proposition.
\vskip5pt

Let $\mathcal C$ be a full subcategory of an abelian category closed under extensions and summands and suppose $\Ext_{\mathcal C}^1(A,B)$ forms a set for all $A,B\in \mathcal C$. (The last condition holds in most situations of interest to us. For example, if $\mathcal C$ has enough projective objects or enough injective objects or is $\Ext$-finite.) 
An exact sequence $0\rightarrow A\stackrel{f}\rightarrow B\stackrel{g}\rightarrow C\rightarrow 0$ is said to be an {\bf almost split sequence} if $f$ is left almost split and $g$ is right almost split in $\mathcal C$.

We have the following characterization of almost split sequences in terms of functors:
\begin{prop}[\cite{A}, II 4.4]\label{ass1} Let $\mathcal C$ be a full subcategory of an abelian category closed under extensions and summands and also assume that $\Ext_{\mathcal C}^1(A,B)$ is a set for all $A,B\in \mathcal C$.
Let $0\rightarrow A\stackrel{f}\rightarrow B\stackrel{g}\rightarrow C\rightarrow 0$ be an exact sequence in $\mathcal C$. The following are equivalent:
\begin{enumerate}
\item  $0\rightarrow A\stackrel{f}\rightarrow B\stackrel{g}\rightarrow C\rightarrow 0$ is an almost split sequence.
\item $g$ is minimal right almost split.
\item [(2')] $f$ is minimal left almost split.
\item $g$ is right almost split and $\End(A)$ is local.
\item [(3')] $f$ is left almost split and $\End(C)$ is local.
\item $\cok(-,g)$ is a simple functor and $(-,B)\stackrel{(-,g)}\rightarrow (-,C)\rightarrow \cok(-,g)\rightarrow 0$ is a minimal projective presentation.
\item [(4')] $\cok(f,-)$ is a simple functor and $(B,-)\stackrel{(f,-)}\rightarrow (A,-)\rightarrow \cok(f,-)\rightarrow 0$ is a minimal projective presentation.
\end{enumerate}
 \end{prop}

Last, we will focus on the specified condition that $\mathcal C$ is an abelian category. The following Lemma is also well-known: 
\begin{lem}[\cite{BLP}, 2.1]\label{mras mono}
 Let $\mathcal C$ be an abelian category. Then
 \begin{enumerate}
\item The morphism $f: A\rightarrow B$  is a minimal right almost split monomorphism if and only if $\cok f$ is simple and $B$ is its projective cover.
\item The morphism $g: A\rightarrow B$ is a minimal left almost split epimorphism if and only if $\Ker f$ is simple and $A$ is its injective hull.
\end{enumerate}
\end{lem}

With that, we have a characterization of minimal right almost split morphisms:
\begin{lem}\label{dich}
If $f: A\rightarrow B$ is a minimal right almost split morphism in abelian category $\mathcal C$, then exactly one of the following two will hold:
\begin{enumerate}
\item $f$ is an epimorphism and $B$ is not projective.
\item  $f$ is a monomorphism and $B$ is projective.
\end{enumerate}
\end{lem}
\begin{proof}
If $f$ is an epimorphism, then since it is minimal right almost split, it is clear that $B$ is not projective.

Now assume $f$ is not an epimorphism, then the monomorphism $i:\Ima f\hookrightarrow B$ is not a split epimorphism. Since $f$ is right almost split, $i$ factors through $f$, which implies that $\Ima f$ is a direct summand of $A$. It follows that $A\simeq \Ima f$ because $f$ is right minimal. Hence $f$ is a monomorphism and by \ref{mras mono} the module $B$ is projective. \end{proof}

\section{Existence and uniqueness theorems}

In this section, let $\mathcal C$ be an additive category. For any morphism $f: X\rightarrow Y$, we have an induced exact sequence of functors:

$$
\xymatrix{(-,X)\ar[r]^{(-,f)}&(-,Y)\ar[r]&F\ar[r]&0.}
$$
Hence the cokernel functor $F$ is finitely presented. We are going to show when the socle of the cokernel functor gives rise to the existence of the unique minimal right determiner of $f$.  
%
%
%
%
%
%
\subsection{Almost factorization} 
We will recall the notion of almost factorization, which will always contribute towards finding the minimal right determiners, however sometimes it will not be enough.
\begin{defn}\label{aft defn}\cite[XI.2]{ARS}
Let $f:X\rightarrow Y$ be a morphism in an additive category $\mathcal C$. A morphism $\alpha: Z\rightarrow Y$ is said to {\bf almost factor through $f$} if $Z$ is indecomposable and it satisfies the following:\\ 
$\bullet$ $\alpha$ cannot factor through $f$.\\
$\bullet$ For any $U$ and any morphism $h:U\rightarrow Z$ in $\rad_{\mathcal C}(U,Z)$, $\alpha h$ can factor through $f$.
\end{defn}
 
  We also say that {\bf an indecomposable object $Z$ can almost factor through $f$} if there is a morphism $\alpha: Z\rightarrow Y$ which can almost factor through $f$.
\vskip5pt
We will also use the following functorial description of almost factorization later. 

\begin{lem}\label{aft}
Let $Z$ be an indecomposable object in an additive category $\mathcal C$. The simple functor $(-,Z)/\rad(-,Z)$ is a subfunctor of $\cok(-,f)$ if and only if $Z$ can almost factor through $f$.
\end{lem}

\begin{proof} ``if part'':  Suppose $Z$ can almost factor through $f$. Then there exists a morphism $\alpha:Z\rightarrow Y$ which can almost factor through $f$. By the definition of almost factorization, $\Ima(-,\alpha)\not\subseteq \Ima(-,f)$ and $(-,\alpha)(i(\rad(-,Z)))\subseteq \Ima(-,f)$ i.e. $\pi (-,\alpha) \neq 0$ and $\pi (-,\alpha) i = 0$ in the following diagram

$$
\xymatrix{0\ar[r]&\rad(-,Z)\ar@{-->}[d]\ar[r]^i&(-,Z)\ar[r]^-\rho\ar[d]^{(-,\alpha)}&(-,Z)/\rad(-,Z)\ar[r]\ar@{-->}[d]^\iota&0\\
0\ar[r] &\Ima(-,f)\ar[r]&(-,Y)\ar[r]^-\pi&\cok(-,f)\ar[r]&0.
}
$$

So there is a non-zero morphism $\iota: (-,Z)/\rad(-,Z)\rightarrow \cok(-,f)$ making the right square commutative. Since $(-,Z)/\rad(-,Z)$ is a simple functor, $\iota$ must be a monomorphism and hence $(-,Z)/\rad(-,Z)$ is a subfunctor of $\cok(-,f)$.
\vskip 5pt

``only if part'': Suppose $\iota: (-,Z)/\rad(-,Z)\rightarrow \cok(-,f)$ is a monomorphism. Since $(-,Z)$ is projective, there is a morphism $\eta: (-,Z)\rightarrow (-,Y)$ lifting the epimorphism $\pi:(-,Y)\rightarrow \cok (-, f)$. i.e. $\iota\rho=\pi\eta$. By Yoneda Lemma, $\eta=(-,\alpha)$ for some $\alpha: Z\rightarrow Y$. It follows that $\pi(-,\alpha)\neq0$ and $\pi(-,\alpha)i=0$, for the inclusion $i:\rad(-,Z)\rightarrow (-,Z)$. So $\alpha$ can almost factor through~$f$. \end{proof}

Next, we discuss the relation between almost factorization and minimal right determiners and show that all objects which almost factor through $f$ must be in any minimal right determiner, however sometimes more objects have to be included in order to obtain a minimal right determiner.

\begin{lem} \label{common}
Let $\mathcal D$ be a right determiner of a morphism $f:X\rightarrow Y$. If an indecomposable object $Z$ can almost factor through $f$ then $Z\in \mathcal D$.
\end{lem}
\begin{proof} Suppose $\alpha: Z\rightarrow Y$ can almost factor through $f$ and $Z\notin \mathcal D$.
Because $\alpha$ cannot factor through $f$ and $\mathcal D$ is a right determiner of $f$, by Corollary \ref{contrapositive2} there is some indecomposable object $U\in \mathcal D$ and a morphism $\beta: U\rightarrow Z$ such that $\alpha\beta$ cannot factor through $f$. Since $Z\notin \mathcal D$, and hence $\beta\in\rad_\mathcal C(U,Z)$ . It contradicts the fact that $\alpha$ can almost factor through~$f$. \end{proof}

Consequently, we have the following crucial observation:
\begin{lem}\label{exist min}
Let $\mathcal D_0$ is the subcategory additively generated by the indecomposable objects which can almost factor through $f$. If $f$ can be right determined by $\mathcal D_0$, then 
\begin{enumerate}
\item $\mathcal D_0=\mathop{\bigcap}\limits_{\mathcal D\in\mathcal S}\mathcal D$, where $\mathcal S$ is the poset of all the right determiners of $f$.
\item $\mathcal D_0$ is the unique minimal right determiner.
\end{enumerate}
\end{lem}
\begin{proof}  If $\mathcal D_0=\bf 0$, then $f$ is a split epimorphism. Hence the statements hold automatically. Now assume $\mathcal D_0$ contains at least one non-zero object.  Since $f$ can be right determined by $\mathcal D_0$, it follows that $\mathcal D_0\in\mathcal S$. On the other hand, by Lemma \ref{common}, $\mathcal D_0\subseteq\mathop{\bigcap}\limits_{\mathcal D\in\mathcal S}\mathcal D$. So $\mathcal D_0=\mathop{\bigcap}\limits_{\mathcal D\in\mathcal S}\mathcal D$ 
which is the unique minimal right determiner. \end{proof}

\begin{defn}\label{semi strong def}
We say that a minimal right determiner $\mathcal D$ of a morphism $f$ is {\bf semi-strong}, if all the non-zero indecomposable objects in $\mathcal D$ can almost factor through $f$. By convention, $\mathcal D=\bf 0$ is semi-strong. 
\end{defn}

\begin{rmk}\label{ss unique}
According to Lemma \ref{exist min}, if $f$ has a semi-strong minimal right determiner $\mathcal D$, then this minimal right determiner is unique. However, the converse is not true (see Example \ref{wedge close}).
\end{rmk}


\begin{defn}\label{anf defn}
Let $f:X\rightarrow Y$ be a morphism in an additive category $\mathcal C$, we say a morphism 
$g:X'\rightarrow Y$ is an {\bf absolute non-factorization} of $f$ if $X'$ is indecomposable, $g$ cannot factor through $f$ and for any indecomposable $T\in\mathcal C$ and any morphism $\alpha: T\rightarrow X'$, $g\alpha$ cannot almost factor through $f$.
\end{defn}


%
%

An example of a morphism and an absolute non-factorization is given in Example \ref{non-strong ex}. 
 We now give a typical construction of a sequence of morphisms using absolute non-factorizations, which will be used in proofs later. 

\begin{lem} \label{chain}
Let $f:X\rightarrow Y$ be a morphism which has a minimal right determiner $\mathcal D$. 
\begin{enumerate}
\item If for an indecomposable object $X'$ the morphism $f': X^\prime\rightarrow Y$ can neither factor through nor almost factor through $f$, then there is an indecomposable object $Z\in \mathcal D$ and a non-isomorphism $g: Z\rightarrow X^\prime$ such that $f^\prime g$ cannot factor through $f$. 
\item  If $f': X^\prime\rightarrow Y$ is an absolute non-factorization of $f$, then there is an (infinite) sequence of non-isomorphisms $\alpha_i$:
$$
\cdots Z_2\stackrel{\alpha_2}\rightarrow Z_1\stackrel{f^\prime\alpha_1}\rightarrow Y
$$
where $Z_i$'s are indecomposable objects in $\mathcal D$, such that the morphism $f'\alpha_1\cdots\alpha_i$ cannot factor through $f$ for $\forall i>0$. 
\end{enumerate}
\end{lem}
\begin{proof}
 (1) By assumption, $f'$ can neither factor through $f$ nor almost factor through $f$. So there is some indecomposable object $U$ and a non-isomorphism $u:U\rightarrow X^\prime$ such that $f^\prime u$ cannot factor through $f$. Since $\mathcal D$ is a minimal right determiner, there is an indecomposable object $Z\in \mathcal D$ and $v: Z\rightarrow U$ such that $f^\prime uv$ cannot factor through $f$.  Hence $g=uv$ is the desired morphism, which is not an isomorphism.

(2) Apply (1) iteratively. 
\end{proof}

\subsection{The uniqueness of minimal right determiner} In the following we give a criterium of the uniqueness of the minimal right determiner. Recall that in a Krull-Schmidt category, $\End(Z)$ is a local ring when $Z$ is an indecomposable object. $\Hom(Z,Y)$ is an $\End(Z)$-module. If $f: X\rightarrow Y$ is a morphism, then $\Ima \Hom(Z,f)$ is an $\End(Z)$-submodule of $\Hom(Z,Y)$. Since $\rad \End(Z)$ is the unique maximal ideal of $\End(Z)$, $(\rad \End(Z))\Hom(Z,Y)$ is a submodule of $\Hom(Z,Y)$.

\begin{thm}\label{unique}
Let $\mathcal C$ be a Krull-Schmidt category and let $\mathcal D$ be a minimal right determiner of $f:X\rightarrow Y$. Suppose for each indecomposable object $Z\in \mathcal D$ there is a natural number $n$ such that $(\rad \End(Z))^n\Hom(Z,Y)\subseteq \Ima\Hom(Z,f)$. Then $\mathcal D$ is the unique semi-strong minimal right determiner of $f$.

%
\end{thm}

\begin{proof}
According to Remark $\ref{mrd not minimal}$, without loss of generality we can assume $f$ is right minimal. If $f$ is an isomorphism, then $\mathcal D=\bf 0$ is the semi-strong minimal right determiner. 

Now assume $f$ is not an isomorphism and hence  $\mathcal D\neq \bf 0$. Let $Z$ be an indecomposable object in $\mathcal D$. Then due to Lemma \ref{in mrd}, there exists a morphism $g:T\rightarrow Y$ which cannot factor through $f$ such that
(i) For any indecomposable $Z^\prime\in \mathcal D$, $Z\not\simeq Z$ and any morphism $\varphi:Z^\prime\rightarrow T$, $g\varphi$ can factor through $f$.
(ii) There is a morphism $\psi_0:Z\rightarrow T$ such that $g\psi_0$ cannot factor through $f$.

Now consider the morphism $g\psi_0$. If $g\psi_0$ can almost factor through $f$, then we finish. Otherwise, by Lemma \ref{chain} there is some indecomposable $Z^\prime\in \mathcal D$ and a non-isomorphism $\psi_1:Z^\prime\rightarrow Z$ such that the composition $g\psi_0\psi_1$ cannot factor through $f$. Since $g$ satisfies condition (i), we have $Z^\prime\simeq Z$.

Consider the morphism $g\psi_0\psi_1: Z\rightarrow Y$. If it can almost factor through $f$, we finish. Otherwise continue the same procedure.

Notice that $\psi_1\cdots\psi_{i-1}\psi_i\in (\rad \End(Z))^{i}$. By the assumption that $(\rad \End(Z))^n\Hom(Z,Y)\subseteq\Ima\Hom(Z,f)$, this procedure must stop before $n$ steps. Hence $Z$ can almost factor through $f$. Therefore $\mathcal D$ is semi-strong and by Remark \ref{ss unique}, it is unique.
\end{proof}

\begin{rmk}\label{D} Here are some special cases when the conditions of the theorem hold:\\
\indent (a) $\mathcal C$ is a Krull-Schmidt category and for each indecomposable object $Z$, $\rad \End(Z)$ is nilpotent.\\
\indent (b) $\mathcal C$ is a Krull-Schmidt category and for each indecomposable object $Z$, $\End(Z)$ is artinian.\\
\indent (c) $\mathcal C$ is a $\Hom$-finite additive $k$-category having split idempotents.\\
\indent (d) $\mathcal C$ is a $\Hom$-finite abelian $k$-category.
\end{rmk}
In particular, we have:
\begin{cor}
Let $\Lambda$ be an artin algebra and $\mathcal C=\Lambda\textendash\modd$ be the category of finitely generated $\Lambda$-modules. Let $f$ be a morphism in $\mathcal C$. Then the minimal right determiner of $f$ is~unique.
\end{cor}

\subsection{Strong minimal right determiner}
\begin{prop}\label{s is ss}
Let $f:X\rightarrow Y$ be a morphism in an additive category $\mathcal C$. If $f$ has no absolute non-factorization, then $f$ has a unique semi-strong minimal right determiner.
\end{prop}
\begin{proof}
Let $\mathcal D_0$ be the full subcategory additively generated by all indecomposable objects which can almost factor through $f$. Since $f$ has no absolute non-factorizations, for any morphism $f':X'\rightarrow Y$ which cannot factor through $f$, since $f'$ restricts on each indecomposable summand of $X'$ is an absolute non-factorization, there is some indecomposable $Z$ and a morphism $\alpha: Z\rightarrow X'$ such that $f'\alpha$ can almost factor through $f$. (Therefore $f'\alpha$ does not factor through $f$.)  So by Proposition \ref{contrapositive}, $f$ is right determined by $\mathcal D_0$. Hence, by Lemma \ref{exist min}, $\mathcal D_0=\mathcal D$ is the unique minimal right determiner of $f$. In particular, $\mathcal D$ is semi-strong.
\end{proof}

\begin{defn}\label{strong defn} A minimal right determiner $\mathcal D$ of a morphism $f: X\rightarrow Y$ is called {\bf strong} if $f$ has no absolute non-factorization as in the definition \ref{anf defn}. \end{defn}

Proposition \ref{s is ss} shows that strong minimal right determiners are semi-strong.

\begin{rmk}
\ \\
(1) Semi-strong minimal right determiners are not necessarily strong. (See example \ref{non-strong ex}).\\
(2) From Remark \ref{ss unique}, if $f$ has a strong minimal right determiner, then it is the unique minimal right determiner of $f$.
\end{rmk}

The notion of strong minimal right determiner follows from Auslander's idea that the minimal right determiner is given by the socle of the cokernel functor. In fact, this characterizes all the strong minimal right determiner. We are going to prove Theorem \ref{exist}.



%
\begin{thm} \label{exist1}
Let $f:X\rightarrow Y$ be a morphism in a Krull-Schmidt 
category $\mathcal C$. Suppose $\soc \cok(-,f)$ is essential and $\{(-,D_i)/r(-,D_i)\}$ is a complete set of pairwise non-isomorphic direct summands of $\soc\cok(-,f)$. Let $\mathcal D=\add\{D_i\}$, then
\begin{enumerate}
\item $\mathcal D$ is the minimal right determiner of $f$.
\item $\mathcal D$ is strong.
\end{enumerate} 
\end{thm}

\begin{proof} (1)
We first show that the subfunctor $\Ima(-,f)$ of $\Hom(-,Y)$ is determined by $\mathcal D$.  Let $G$ be a subfunctor of $\Hom(-,Y)$. By Definition \ref{functor det obj def} It suffice to show that if $G(D_i)\subseteq \Ima(D_i, f)$ for all $D_i\in \mathcal D$, then $G$ is a subfunctor of $\Ima(-,f)$.

Since $G\subseteq\Hom(-,Y)$, it follows that the functor $G/(G\cap \Ima(-,f))\subseteq \Hom(-,Y)/\Ima(-,f)\simeq \cok(-,f)$.
%
 If $G/(G\cap\Ima(-,f))\neq 0$, then because $\soc \cok(-,f)$ is essential, there is some simple subfunctor $S_i\simeq (-,D_i)/r(-,D_i)$ such that $S_i\subseteq G/(G\cap\Ima(-,f))$ by Proposition \ref{ess equiv}. However
$$
0\neq S_i(D_i)\subseteq G(D_i)/(G(D_i)\cap\Ima(D_i,f))
$$
which implies that $G(D_i)\not\subseteq \Ima(D_i, f)$ which is a contradiction. Hence $G/(G\cap\Ima(-,f))=0$ and so $G\subseteq\Ima(-,f)$.

Therefore the subfunctor $\Ima(-,f)$ of $\Hom(-,Y)$ is determined by $\mathcal D$. By Proposition \ref{det fun}, $f$ is right determined by $\mathcal D$. By Proposition \ref{aft}, it follows that each $D_i\in\mathcal D$ can almost factor through $f$. 
Hence by Lemma \ref{exist min}, $\mathcal D$ is the minimal right determiner.

%
(2) By definition \ref{strong defn}, we need to show that for each morphism $g: T\rightarrow Y$  which cannot factor through $f$, there exists some morphism $h: Z\rightarrow T$ such that $gh$ can almost factor through $f$.

Let $g: T\rightarrow Y$ be a morphism which cannot factor through $f$, then the composition
$$
\xymatrix{(-,T)\ar[r]^{(-,g)}&(-,Y)\ar[r]^-\pi&\cok(-,f)}
$$
is non-zero. Then $\Ima(\pi\circ (-, g)) \stackrel{s}\hookrightarrow \cok(-,f )$ is a non-zero subfunctor. So it contains a simple subfunctor $S_i\simeq (-,D_i)/r(-,D_i)$ by our assumption that $\soc\cok(-,f)$ is essential. So $S_i$ is a subfunctor of $\Ima(\pi\circ(-, g))$  and we have the commutative diagram:

$$
\xymatrix{ (-,D_i)\ar[r]^{\rho_i}\ar@{-->}[d]^\eta&S_i\ar@{^(->}[d]^\iota\ar[r]&0\\
(-,T)\ar[r]^{\pi\circ(-,g)} &\Ima \pi(-,g)\ar[r]&0 
}
$$

By Yoneda Lemma, there is $h: D_i\rightarrow T$ such that $\eta= (-, h)$. It follows from Lemma \ref{aft} that $gh$ can almost factor through $f$.

Therefore $\mathcal D$ is the strong minimal right determiner of $f$.
\end{proof}

\begin{prop}\label{exist3}
Let $f:X\rightarrow Y$ be a morphism in a Krull-Schmidt category $\mathcal C$. If $f$ has a strong minimal right determiner then $\soc\cok(-,f)$ is essential.  
\end{prop}
\begin{proof}
Suppose $\mathcal D = \add \{D_i\}$ is a strong minimal right determiner of $f$. It suffices to show that any non-zero subfunctor $H\subseteq \cok(-, f)$ has a simple subfunctor.

Suppose $T$ is an object such that $H(T)\neq 0$. Then there is a morphism $g:T\rightarrow Y$ which cannot factor through $f$. By our assumption that $\mathcal D$ is strong, there is some $D_i\in \mathcal D$ and a morphism $h: D_i\rightarrow T$ such that $gh$ can almost factor through $f$. The morphism
\begin{eqnarray*}
H(h): H(T)&\rightarrow& H(D_i)\\
g&\mapsto& gh
\end{eqnarray*}
implies that $H(D_i)\neq 0$. By Yoneda Lemma, there is a morphism $\eta: (-,D_i)\rightarrow H$ such that  $\eta_{D_i}$ sends $id_{D_i}$ to $gh$. The fact that $gh$ can almost factor through $f$ implies that  $\eta i=0$ in the following commutative diagram.
$$
\xymatrix{0\ar[r]&\rad(-,D_i)\ar[r]^i&(-,D_i)\ar[r]^-\rho\ar[d]^\eta&(-,D_i)/\rad(-,D_i)\ar[r]\ar@{-->}[dl]^\iota&0\\
             &&H}
$$

So there is a morphism $\iota: (-,D_i)/\rad(-,D_i)\rightarrow H$ such that $\iota \rho=\eta$. Since $\eta\neq0$, $\iota\neq 0$. Hence, it must be a monomorphism.
Therefore, we know that $H$ contains a simple subfunctor $(-,D_i)/\rad(-,D_i)$. Hence $\soc \cok(-,f )$ is essential.  
\end{proof}

{\bf Proof of Theorem \ref{exist}:} If $\soc\cok(-,f)$ is essential, then by Theorem \ref{exist1}, $f$ has a strong minimal right determiner. Conversely, if $f$ has a minimal right determiner, then by Proposition \ref{exist3}, $\soc\cok(-,f)$ is essential. \hfill$\qed$

\section{Minimal right determiner formula} 
Let $\mathcal C$ be an abelian category and $D^b(\mathcal C)$ be its bounded derived category.  The category $\mathcal C$ is called {\bf hereditary} if $\Ext_{\mathcal C}^i(A,B):=\Hom_{D^b(\mathcal C)}(A,B[i])=0$ for all $i\geq 2$ and $A,B\in \mathcal C$. In case the category $\mathcal C$ has enough projective objects or enough injective objects we can compute these Ext groups using projective or injective resolutions. The elements of the group $\Ext_{\mathcal C}^i(A,B)$ can also be described as Yoneda extensions which are equivalence classes of long exact sequences starting from $B$ and ending in $A$ of length $i$.
It is well-known that $\mathcal C$ is hereditary if and only if $\Ext_{\mathcal C}^1(X,-)$ preserves epimorphisms for all $X\in \mathcal C$ (\cite{RV}, Lemma A.1).

Throughout this section, {\bf assume $\mathcal C$ is a hereditary $\Hom$-finite abelian  $k$-category}. In particular, $\mathcal C$ is Krull-Schmidt. A class of typical examples of hereditary $\Hom$-finite abelian  $k$-categories will be the categories of finitely presented representations of strongly locally finite quivers which will be discussed in next section. Also there is a classification of noetherian hereditary abelian categories in \cite{RV}. We do not require our category $\mathcal C$ to have Serre duality.

By Theorem \ref{unique}, for every morphism $f\in\mathcal C$, if $f$ has a minimal right determiner, it is the unique minimal right determiner. Our goal is to answer the following questions:\\
(1) When does a morphism have a minimal right determiner in terms of properties of $\Ker f$ and $\cok f$? \\
(2) If a morphism $f$ has a minimal right determiner, how to compute it in terms of $f$?

 First of all, we prove this Lemma which will be heavily used in the following.
\begin{lem}\label{pullback}
Suppose we have the following commutative diagram with exact rows:
$$
\xymatrix{0\ar[r]&A_1\ar[r]^{f_1}\ar[d]^u&B_1\ar[r]^{g_1}\ar[d]^v&C_1\ar[d]^w\ar[r]&0\\
          0\ar[r]&A_2\ar[r]^{f_2}&B_2\ar[r]^{g_2}&C_2}
$$
and $u$ is a split monomorphism. Then $w$ can factor through $g_2$ if and only if $g_1$ is a split epimorphism.
\end{lem}
\begin{proof} If $g_1$ is a split epimorphism, then there is a morphism $s: C_1\rightarrow B_1$ such that $g_1s=id_{C_1}$. So $g_2vs=wg_1s=w$. Hence, $w$ can factor through $g_2$.

Conversely, suppose $w=g_2t$.
$$
\xymatrix{0\ar[r]&A_1\ar[r]^{f_1}\ar[d]^u&B_1\ar[r]^{g_1}\ar[d]^v\ar@{-->}[dl]_s&C_1\ar@{-->}[dl]_t\ar[d]^w\ar[r]&0\\
          0\ar[r]&A_2\ar[r]^{f_2}&B_2\ar[r]^{g_2}&C_2}
$$

Then $g_2(v-tg_1)=g_2v-g_2tg_1=g_2v-wg_1=0$. So $v-tg_1$ can factor through $f_2$. i.e. there is a morphism $s: B_1\rightarrow A_2$ such that $f_2s=v-tg_1$. Hence $f_2(u-sf_1)=f_2u-f_2sf_1=f_2u-(v-tg_1)f_1=f_2u-vf_1=0$ which implies that $u=sf_1$. Since $u$ is a split monomorphism, $f_1$ is also a split monomorphism and hence $g_1$ is a split epimorphism. \end{proof}

\subsection{Neccessary conditions for the existence of minimal right determiners} Assume that a right minimal non-isomorphism: $f:X\rightarrow Y$ in $\mathcal C$ has a minimal right determiner $\mathcal D$. Since $\mathcal C$ is a $\Hom$-finite abelian category, by Theorem \ref{unique} and Remark \ref{D}, $\mathcal D$ is semi-strong. First, we are going to show that the simple functor corresponding to $D_i$ is finitely presented. Then from a finitely presented simple functor $(-,D_i)/\rad(-,D_i)$, we are going to construct a right almost split morphism ending in $D_i$, which reveals a local almost split structure.

\begin{lem}
Let  $\mathcal C$ is a hereditary $\Hom$-finite abelian  $k$-category. Suppose a morphism $f:X\rightarrow Y$ in $\mathcal C$ has a minimal right determiner $\mathcal D$. Then for each $D_i\in\mathcal D$, the simple functor $(-,D_i)/\rad(-,D_i)$ is finitely presented.
\end{lem}
\begin{proof}
%

Since $\mathcal D$ is semi-strong, each $D_i\in\mathcal D$ can almost factor through $f$. By Lemma \ref{aft}, $(-,D_i)/\rad(-,D_i)$ is a subfunctor of $\cok(-,f)$. 

We know that simple functors are finitely generated and $\cok(-,f)$ is finitely presented. Since $\mathcal C$ is an abelian category, by Lemma \ref{coh}, $(-,D_i)/\rad(-,D_i)$ is finitely presented. \end{proof}
\vskip5pt

 Hence there exists a minimal projective presentation of $(-,D_i)/\rad(-,D_i)$:

$$
(-,M)\stackrel{(-,t)}\rightarrow (-,D_i)\rightarrow (-,D_i)/\rad(-,D_i)\rightarrow 0.
$$

It follows that $M\stackrel{t}\rightarrow D_i$ is a minimal right almost split morphism. By Lemma \ref{dich}, it splits into two cases: (1) $D_i$ is non-projective and $t$ is an epimorphism. (2) $D_i$ is projective and $t$ is a monomorphism. We are going to discuss each of them separately.

Suppose $D_i=N$ is non-projective. We have an exact sequence:
$$
0\rightarrow L\stackrel{s}\rightarrow M\stackrel{t}\rightarrow N\rightarrow 0.
$$
By Corollary \ref{ass1}, it is an almost split sequence and in particular, $L$ is indecomposable


Since $N$ can almost factor through $f$, there is a morphism $w: N\rightarrow Y$ such that $w$ cannot factor through $f$, but $wt$ factors through $f$. Hence we have the following commutative diagram:

$$
\xymatrix{0\ar[r]&L\ar[r]^s\ar[d]^u&M\ar[r]^t\ar[d]^v  &N\ar[r]\ar[d]^w &0\\
          0\ar[r]&\Ker f\ar[r]^i&X\ar[r]^f&Y}
$$

\begin{lem}\label{ker|}
In the commutative diagram above, $L$ is a direct summand of $\Ker f$.
 \end{lem}
 \begin{proof} Otherwise, $u$ is not a split monomorphism. Since $s$ is left almost split, there is a morphism $\beta:M\rightarrow \Ker f$ such that $\beta s=u$. Since $(v-i\beta)s=0$, there is a morphism $\gamma: N\rightarrow X$ such that $v-i\beta=\gamma t$. Now, $wt=fv=f(\gamma t+i\beta)=f\gamma t$, and since $t$ is an epimorphism, it follows that $w=f\gamma$ which contradicts with Lemma \ref{pullback} \end{proof}


 Suppose $D_i=P$ is projective. Because $t: M\rightarrow P$ is minimal right almost split monomorphism, by Lemma \ref{mras mono}, $P$ is the projective cover of the simple object $S\simeq\cok t$. 

\begin{lem}
The simple object $S$ is a subobject of $\cok f$.
\end{lem}

\begin{proof} For each indecomposable projective object $P\in \mathcal D$, since $P$ can almost factor through $f$, there is a morphism $\alpha$ which cannot factor through $f$ such that the composition $\alpha \iota$ factors through $f$.
$$
\xymatrix{0\ar[r]&\rad P\ar[r]^\iota\ar@{-->}[d]&P\ar[r]\ar[d]^\alpha&S\ar[r]\ar[d]^s&0\\
           &X\ar[r]^f&Y\ar[r]&\cok f\ar[r]&0.}
$$
This induces a morphism $s: S\rightarrow \cok f$ where $S\simeq P/\rad P$ is a simple object. Because $\alpha$ cannot factor through $f$, the morphism $s\neq0$. Hence it must be a monomorphism. i.e. $S$ is a simple subobject of $\cok f$. \end{proof}

\vskip10pt

To summarize, so far we obtain the following consequences:

\begin{prop}\label{local ar}
Let $f$ be a right minimal non-isomorphism in $\mathcal C$ having a minimal right determiner $\mathcal D$. Then
\begin{enumerate}
\item For each non-projective object $N\in \mathcal D$, $N$ is the ending term of an almost split sequence:
$$
0\rightarrow L\rightarrow M\rightarrow N\rightarrow 0
$$
and $L$ is a direct summand of $\Ker f$.
\item For each projective object $P\in \mathcal{D}$, $P$ is the projective cover of some simple object $S\subseteq\cok(f)$.
\item If $f$ is a monomorphism, then $\mathcal D$ contains only projective objects.
\item If $f$ is an epimorphism, then $\mathcal D$ contains only non-projective objects.
\end{enumerate}
\end{prop}

Notice that $\Ker f$ only has finitely many indecomposable summands. So we have the following useful Corollary.
\begin{cor}\label{ker finite}
 A minimal right determiner $\mathcal D$ contains only finitely many isomorphism classes of indecomposable non-projective  objects.
 \end{cor}
 
 This fact gives rise to a typical way to analyze indecomposable objects in a minimal right determiner which will be demonstrated in the following Lemmas.
 
\begin{lem}\label{non-proj finite}
 Let $f:X\rightarrow Y$ be a morphism having a minimal right determiner $\mathcal D$. Suppose $$
\cdots Z_2\stackrel{\alpha_2}\rightarrow Z_1\stackrel{\alpha_1}\rightarrow Y
$$  is an infinite sequence of morphisms with $Z_i$'s indecomposable objects in $\mathcal D$ and $\alpha_i$'s are non-isomorphisms such that $\alpha_1\alpha_2\cdots\alpha_k\neq 0$ for any $k$. Then there is an integer $N$ such that $Z_i$ are all projective for $i>N$.
 \end{lem}
 \begin{proof}
 Suppose there are infinitely many non-projective $Z_i$'s in such a sequence of morphisms. Since $\mathcal D$ only contains finitely many non-isomorphic indecomposable non-projective objects, there must be an indecomposable object $N\in\mathcal D$ such that $Z_j\simeq N$ for infinitely many $j$, which means $\rad\End(Z)$ is not nilpotent. This contradicts with the category $\mathcal C$ being $\Hom$-finite. 
 \end{proof}
 
 \begin{lem} \label{extend to aft}
Let $f:X\rightarrow Y$ be a morphism having a minimal right determiner $\mathcal D$. Suppose there is a commutative diagram
$$
\xymatrix{U\ar[d]^s\ar@{->>}[r]^t&V\ar[d]^g\\X\ar[r]^f&Y}
$$
where $t$ is an epimorphism and $g$ cannot factor through $f$. Then 
\begin{enumerate}
\item For any projective object $P$ and any morphism $p:P\rightarrow V$, $gp$ factors through $f$.
\item There is an indecomposable non-projective object $Z\in\mathcal D$ and a morphism $\alpha: Z\rightarrow V$ such that $g\alpha$ can almost factor through $f$.
\end{enumerate}
 \end{lem}
 \begin{proof}
 $(1)$ Since $t$ is an epimorphism, $p$ factors through $t$. Therefore, $gp$ factors through $f$.\\
 $(2)$ Without loss of generality, we can assume that $V$ is indecomposable. Suppose $g$ is an absolute non-factorization of $f$. Then by Lemma \ref{chain}, there is a sequence of of non-isomorphisms~$\alpha_i$:
$$
\cdots Z_2\stackrel{\alpha_2}\rightarrow Z_1\stackrel{g\alpha_1}\rightarrow Y
$$
where $Z_i$'s are indecomposable objects in $\mathcal D$, such that the morphism $g\alpha_1\cdots\alpha_n$ cannot factor through $f$ for $\forall n$. By Lemma \ref{non-proj finite}, there is an indecomposable projective $Z_N\in\mathcal D$ and a morphism $\alpha=\alpha_1\cdots\alpha_N: Z_N\rightarrow V$ such that $g\alpha$ cannot factor through $f$, which contradicts (1). Hence there is an indecomposable object $Z$ and a morphism $\alpha:Z\rightarrow V$ such that $g\alpha$ almost factors through $f$. Then because $g\alpha$ cannot factor through $f$, due to (1) again, $Z$ is non-projective. 
 \end{proof}

\begin{prop}\label{non-strong chain}
  If $f:X\rightarrow Y$ has a minimal right determiner $\mathcal D$ which is not strong, then there is an infinite sequence of morphisms 
 $$
\cdots P_3\stackrel{\alpha_3}\rightarrow P_2\stackrel{\alpha_2}\rightarrow P_1\stackrel{\alpha_1}\rightarrow Y
$$
with $P_i$'s indecomposable projective objects in $\mathcal D$ and $\alpha_i$'s are proper monomorphisms ($i>1$) such that $\alpha_1\alpha_2\cdots\alpha_k$ cannot factor through $f$ for any $k$. In particular $\mathcal D$ contains infinitely many indecomposable objects.
\end{prop}
\begin{proof} Since $\mathcal D$ is not strong, there is a morphism $f': X^\prime\rightarrow Y$ which is an absolute non-factorization of $f$ . By Lemma \ref{chain}, there is an infinite sequence of non-isomorphisms $\varphi_i$:
$$
\cdots Z_2\stackrel{\varphi_2}\rightarrow Z_1\stackrel{f^\prime\varphi_1}\rightarrow Y
$$
with $Z_i$'s indecomposable objects in $\mathcal D$ such that $f'\varphi_1\cdots\varphi_k$ cannot factor through $f$ for all $k$. By Corollary \ref{non-proj finite}, there is an integer $N$ such that $Z_i$ is projective when $i>N$. Take $Z_{N+i}=P_i$, we have a sequence of morphisms:
  $$
\cdots P_3\stackrel{\alpha_3}\rightarrow P_2\stackrel{\alpha_2}\rightarrow P_1\stackrel{\alpha_1}\rightarrow Y
$$
with $P_i$'s indecomposable projective objects in $\mathcal D$, $\alpha_i=\varphi_{N+i+1}$ for $i>1$ and $\alpha_1=f'\phi_1\cdots\phi_{N+1}$ are non-isomorphisms. Since $\mathcal C$ is hereditary, $\alpha_i$'s are proper monomorphisms for $i>1$. Hence $\mathcal D$ contains infinitely many indecomposable objects. \end{proof}

\vskip10pt

Next, we are going to show that if $f$ has a minimal right determiner, then each indecomposable direct summand of $\Ker f$ and simple subobject of $\cok f$ should behave like what was described in Theorem \ref{local ar} (1) and (2).

\begin{prop}\label{ker}
Let $f:X\rightarrow Y$ be a right minimal non-monomorphism. If $f$ has a minimal right determiner $\mathcal D$ then each indecomposable direct summand $K$ of $\Ker f$ is the starting term of an almost split sequence.
\end{prop}

\begin{proof} 
Take the left minimal version (dual of right minimal version, see \ref{r minimal}) of the composition of the inclusions: $u_1: K\hookrightarrow \Ker f\hookrightarrow X$, we obtain a monomorphism: $s_1:K\hookrightarrow X_1$ where $X_1$ is the direct sum of all indecomposable summands of $X$ which intersects $K$ nontrivially. Hence, we have the following commutative diagram:

$$
\xymatrix{0\ar[r]&K\ar[r]^{s_1}\ar[d]^{u_1}&X_1\ar[r]^{t_1}\ar[d]^{v_1}&Y_1\ar[r]\ar[d]^{w_1}&0\\
          0\ar[r]&\Ker f\ar[r]^i&X\ar[r]^f&Y}
$$

By the assumption that $f$ is right minimal, the upper exact sequence does not split. Hence $w_1$ cannot factor through $f$ by Lemma \ref{pullback}.

By Lemma \ref{extend to aft},  there is an indecomposable object $Z\in \mathcal D$ and $\alpha:Z\rightarrow Y_1$ such that the composition $w_1\alpha$ almost factors through $f$.

%
%
%
%
%
%

Pull-back the morphism $\alpha$, we have the following commutative diagram:
 $$
 \xymatrix{\delta:&0\ar[r]&K\ar[r]^{s}\ar@{=}[d]&E\ar[r]^{t}\ar[d]&Z\ar[r]\ar[d]^{\alpha}&0\\
&0\ar[r]&K\ar[r]^{s_1}\ar[d]^{u_1}&E_1\ar[r]^{t_1}\ar[d]^{v_1}&Y_1\ar[r]\ar[d]^{w_1}&0\\
          &0\ar[r]&\Ker f\ar[r]^i&X\ar[r]^f&Y}
 $$
It suffice to show that $\delta$ is an almost split sequence. In fact, for each $\beta:V\rightarrow Z$ which is not a split epimorphism, taking the pull-back again we have:
$$
\xymatrix{0\ar[r]&K\ar@{=}[d]\ar[r]&U\ar[d]\ar[r]&V\ar[d]^\beta\ar[r]&0\\
0\ar[r]&K\ar[r]^{s}\ar@{=}[d]&E\ar[r]^{t}\ar[d]&Z\ar[r]\ar[d]^{\alpha}&0\\
0\ar[r]&K\ar[r]^{s_1}\ar[d]^{u_1 }&E_1\ar[r]^{t_1}\ar[d]^{v_1 }&Y_1\ar[r]\ar[d]^{w_1 }&0\\
          0\ar[r]&\Ker f\ar[r]^i&X\ar[r]^f&Y}
$$
Since $w_1\alpha \beta$ factors through $f$, by Lemma \ref{pullback} the upper exact sequence splits. Hence $\beta$ factors through $t$. Therefore $t$ is minimal right almost split and $\delta$ is an almost split sequence. \end{proof}

\vskip10pt

Next we consider the cokernel.  
First, we analyze the exact sequence involving subobjects of $\cok f$. 

\begin{lem}\label{cd}
Let $f:X\rightarrow Y$ be a right minimal non-epimorphism. For any subobject $C\subseteq \cok f$, there is a commutative diagram:
\begin{align}
\xymatrix{0\ar[r]&X\ar[r]^s\ar@{=}[d]&E\ar[r]^t\ar[d]^g&C\ar[r]\ar[d]^l&0\\
               &X\ar[r]^f\ar[r]&Y\ar[r]^-\pi&\cok f\ar[r]&0.} \tag{*}
\end{align}
where $l$ is the inclusion. 
\end{lem}

\begin{proof}

 We have the following pull-back diagram:

$$
\xymatrix{\delta: &0\ar[r]&\Ima f\ar@{=}[d]\ar[r]&Y_0\ar[d]^{g_1}\ar[r]^p&C\ar@{^(->}[d]^l\ar[r]&0\\
&0\ar[r]&\Ima f\ar[r]^i&Y\ar[r]^-\pi&\cok f \ar[r]&0}
$$

Because $\mathcal C$ is hereditary, $\Ext^1(C,-)$ preserves the canonical epimorphism $j: X\rightarrow \Ima f$. i.e. there is an epimorphism of abelian groups $\Ext^1(C,j):\Ext^1(C,X)\rightarrow \Ext^1(C,\Ima f)$. So the exact sequence $\delta$ can be lifted to an exact sequence $\epsilon: 0\rightarrow X\rightarrow E\rightarrow C\rightarrow 0$ in $\Ext^1(X,C)$. Hence we have the following commutative diagram:

\begin{eqnarray*}
\begin{xy}
(-30,30)*+{\epsilon:}; (-20,30)*+{0}="z11";(0,30)*+{X}="x1"; (20,30)*+{E}="e0"; (40,30)*+{C}="s1"; (60,30)*+{0}="z12"; 
(0,15)*+{X}="x2"; (20,15)*+{Y_0}="y0"; (40,15)*+{C}="s2"; (60,15)*+{0}="z22"; 
(0,0)*+{X}="x3"; (20,0)*+{Y}="y"; (40,0)*+{\cok f}="cok"; (60,0)*+{0}="z32"; 
(10,10)*+{\Ima f}="Im1"; (10,-5)*+{\Ima f}="Im2";
{\ar "z11"; "x1"}; {\ar^s "x1"; "e0"}; {\ar^t "e0"; "s1"}; {\ar "s1"; "z12"}; 
{\ar "x2"; "y0"}; {\ar^p "y0"; "s2"}; {\ar "s2"; "z22"};
{\ar^<(0.3)f "x3"; "y"}; {\ar^\pi "y"; "cok"}; {\ar "cok"; "z32"};
{\ar@{=} "x1"; "x2"}; {\ar^{g_2} "e0";"y0"}; {\ar@{=} "s1";"s2"};
{\ar@{=} "x2"; "x3"}; {\ar^{g_1} "y0"; "y"}; {\ar@{^(->}^l "s2"; "cok"};
 {\ar_j "x2"; "Im1"}; {\ar_i "Im1"; "y0"};
  {\ar_j "x3"; "Im2"}; {\ar_i "Im2"; "y"};
  {\ar@{=} "Im1"; "Im2"};
\end{xy}
\end{eqnarray*}

Let $g=g_1g_2$. Then the first row and the third row give rise to the desired commutative diagram.
\end{proof}

\begin{rmk}
In the commutative diagram $(*)$ above, $g$ cannot factor through $\Ima f$.  
\end{rmk}

\vskip5pt
 We need the following useful Lemma: 

\begin{lem}\label{nonproj}
Suppose $f:X\rightarrow Y$ is a morphism in $\mathcal C$ and $\alpha: N\rightarrow Y$ can almost factor through $f$, where $N$ is an indecomposable non-projective object. Then $\Ima \alpha\subseteq \Ima f$.
\end{lem}

\begin{proof} Notice that in an abelian category an object $P$ is projective if and only if each epimorphism $M\rightarrow P$ splits. So $N$ is not projective implies there is a non-split epimorphism $\pi: M\rightarrow N$.
Since $\alpha$ can almost factor through $f$, $\alpha\pi$ can factor through $f$. Therefore we have the following commutative diagram:
$$
\xymatrix{M\ar[r]^\pi\ar[d]^s&N\ar[d]^\alpha\\
                X\ar[r]^f&Y}
$$

Since $\pi$ is an epimorphism, $\Ima\alpha=\Ima\alpha\pi=\Ima fs\subseteq\Ima f$. \end{proof}

\begin{lem}\label{no N aft}
In the commutative diagram $(*)$, suppose there is an indecomposable object $Z$ and a morphism $\alpha:Z\rightarrow E$ such that $g\alpha$ almost factors through $f$. Then 
\begin{enumerate}
\item $Z$ is projective.
\item $\Top Z$ is a simple subobject of $C$.
\end{enumerate}
\end{lem}

\begin{proof}  
  $(1)$ If $Z$ is non-projective, then by Lemma \ref{nonproj} it follows that $\Ima g\alpha\subseteq \Ima f$. That means $g\alpha$ can factor through $\Ima f$. So $lt\alpha=\pi g\alpha=0$, which implies $t\alpha=0$.  Hence $\alpha$ factors through $s$ and $g\alpha$ factors through $f$ which is a contradiction.
  
  $(2)$ By $(1)$, $Z$ is an indecomposable projective object. Then there is an exact sequence: $\xymatrix{0\ar[r]&\rad Z\ar[r]^\iota&Z\ar[r]&\Top Z\ar[r]&0}$. Since $g\alpha\iota$ can factor through $f$,   $\pi gh\alpha\iota=0$ which means $lt\alpha\iota=0$ and hence $t\alpha \iota=0$. So we have the following commutative diagram:
$$
\xymatrix{0\ar[r]&\rad Z\ar[r]^\iota&Z \ar[d]^\alpha\ar[r]&\Top Z\ar@{-->}[d]^u\ar[r]&0\\
0\ar[r]&X\ar[r]^s&E\ar[r]^t&C\ar[r]&0}
$$
Since $g\alpha$ cannot factor through $f$, $\pi g\alpha\neq 0$ which implies $t\alpha\neq 0$. Therefore $u\neq 0$ and $\Top Z$ is a simple subobject of $C$.
  \end{proof}

\begin{prop}\label{cok}
Let $f:X\rightarrow Y$ be a right minimal non-epimorphism. If $f$ has a minimal right determiner $\mathcal D$, then each simple object $S\subseteq\cok f$ has a projective cover. 
\end{prop}

\begin{proof}
By Lemma \ref{cd}, for each simple object $S\subseteq \cok f$, there is a commutative diagram:
$$
\xymatrix{0\ar[r]&X\ar[r]^s\ar@{=}[d]&E\ar[r]^t\ar[d]^g&S\ar[r]\ar[d]^l&0\\
               &X\ar[r]^f\ar[r]&Y\ar[r]^\pi&\cok f\ar[r]&0.} 
$$

{\bf Claim:} Assume $g$ (when restricted on each indecomposable summand) is not an absolute non-factorization of $f$.

Otherwise, by Lemma \ref{chain}, we have a sequence of morphisms:
 $$
\cdots Z_2\stackrel{\alpha_2}\rightarrow Z_1\stackrel{g\alpha_1}\rightarrow E_0
$$ 
with $Z_i$ indecomposable objects in $\mathcal D$, such that $g\alpha_1\cdots\alpha_k$ cannot factor through $f$ for any $k$.
By Corollary \ref{non-proj finite}, $Z_n$ are indecomposable projective objects whenever $n\geq N$ for some sufficiently large $N$ and $\alpha_n$ are proper monomorphisms for $n>N$. In particular,  morphisms $gh\alpha_1\alpha_2\cdots\alpha_n$ cannot factor through $f$ whenever $n\geq N$. Hence $t\alpha_1\alpha_2\cdots\alpha_n: Z_n\rightarrow S$ is non-zero for all $n\geq N$.   On the other hand, the indecomposable projective object $Z_N$ has a non-zero morphism $t\alpha_1\alpha_2\cdots\alpha_N$ to the simple object $S$ if and only if $\Top Z_N\simeq S$. But $\alpha_{N+1}$ is a proper monomorphism implies that $\Ima \alpha_{N+1}\subseteq \rad Z_N$. Hence $t\alpha_1\alpha_2\cdots\alpha_N\alpha_{N+1}=0$ which is a contradiction. 
The claim is verified.

Therefore there is an indecomposable $Z\in \mathcal D$ and a morphism $\alpha:Z\rightarrow E$ such that the composition $g\alpha$ can almost factor through $f$.

By Lemma \ref{no N aft}, $Z$ is projective and $\Top Z\subseteq S$. Hence it follows that $S\simeq \Top Z$ and $Z$ is the projective cover of $S$.
 \end{proof}

Furthermore, we have the following observation:

\begin{prop} \label{strong essential}
Let $f:X\rightarrow Y$ be a right minimal non-epimorphism. If $f$ has a strong minimal right determiner $\mathcal D$, then $\soc\cok f$ is essential.
\end{prop}
\begin{proof} 
Suppose $\soc\cok f$ is not essential which means there is a non-zero object $C\subseteq\cok f$ such that $\soc C=0$.
Then by Lemma \ref{cd}, there is a commutative diagram:
$$
\xymatrix{0\ar[r]&X\ar[r]^s\ar@{=}[d]&E\ar[r]^t\ar[d]^g&C\ar[r]\ar[d]^l&0\\
               &X\ar[r]^f\ar[r]&Y\ar[r]^-\pi&\cok f\ar[r]&0.} 
$$

Since $\mathcal D$ is strong, there is an indecomposable $Z\in\mathcal D$ and a morphism $\alpha:Z\rightarrow E$ such that $g\alpha$ almost factors through $f$. However, by Lemma \ref{no N aft}, $\Top Z\subseteq C$ which contradicts with the fact $\soc C=0$.
\end{proof}

Proposition \ref{strong essential} does not hold for non-strong minimal right determiner. However,
in most cases we are interested in, any minimal right determiner is strong. We have the following observation:

\begin{prop} \label{all strong}
Let $\mathcal C$ be a $\Hom$-finite, hereditary abelian category. $f:X\rightarrow Y$ has a minimal right determiner $\mathcal D$, then $\mathcal D$ is strong when one of the following conditions are satisfied:
\begin{enumerate}
\item $\mathcal D$ contains finitely many indecomposable objects.
\item $\soc \cok f$ is a finite direct sum of simple objects.
\item $\cok f$  has a projective cover.
\end{enumerate}
\end{prop}

 (1) follows directly from Corollary \ref{non-strong chain}. In particular, in the category of coherent sheaves of a projective space, every minimal right determiner is strong. (2) follows from (1) and Proposition \ref{local ar} (2). We need a general fact in hereditary category before showing (3):

Suppose $C$ is an object in a $\Hom$-finite hereditary abelian category $\mathcal C$ with enough projective objects. Let $S$ be a simple subobject of $C$ ($w:S\rightarrow C$ is the inclusion). Taking the minimal projective resolution of $S$ and $C$, we have the following commutative diagram:

$$
\xymatrix{0\ar[r]&\rad P(S)\ar[d]^u\ar[r]^i&P(S)\ar[d]^v\ar[r]&S\ar@{^(->}[d]^w\ar[r]&0\\
0\ar[r]&P_1\ar[r]^j&P_0\ar[r]&C\ar[r]&0}
$$
Then we have:
\begin{lem} \label{seems easy}
In the commutative diagram above, 
\begin{enumerate}
\item $u$, $v$ are monomorphisms.
\item The left square is a pull-back
\item $u$ is a split monomorphism.
\end{enumerate}
\end{lem}
\begin{proof} (1) Notice that $P(S)$ is indecomposable and $\mathcal C$ is hereditary, $v\neq 0$ implies $v$ is a monomorphism. Hence, so is $u$.

(2) Taking the pull-back of $v$, $j$, we have a morphism $\varphi: \rad P(S)\rightarrow P(S)\times P_1$ such that $a \varphi= i$ and $b \varphi=u$.
$$
\xymatrix{
 \rad P(S) \ar@/_/[ddr]_u \ar@/^/[drr]^i
   \ar@{.>}[dr]|-{\varphi}            \\
  & P(S)\times P_1 \ar[d]^b \ar[r]_a
                 & P(S)\ar[d]_v       \\
  & P_1 \ar[r]^j   & P_0                }
$$
From the pull-back, we know $a$ is a monomorphism. Since $i$ is a monomorphism, so is $\varphi$. Therefore, up to isomorphism, we have 
$$
\rad P(S)\subseteq P(S)\times P_1\subseteq P(S)
$$

 Since $\rad P(S)$ is the maximal submodule of $P(S)$, either $P(S)\times P_1\simeq \rad P(S)$ or $P(S)\times P_1\simeq P(S)$. If $P(S)\times P_1\simeq P(S)$, then $a$ is an isomorphism and hence $v$ factors through $j$ which is a contradiction. So $P(S)\times P_1\simeq \rad P(S)$.

We also have to show that $\varphi$ is an isomorphism. Consider the commutative diagram induced by $a \varphi=i$:

$$
\xymatrix{0\ar[r]&\rad P(S)\ar[d]^\varphi\ar[r]^i&P(S)\ar@{=}[d]\ar[r]^\pi&S\ar[d]^{\psi}\ar[r]&0\\
0\ar[r]&P(S)\times P_1\ar[r]^a&P(S)\ar[r]^\rho&S\ar[r]&0,}
$$
 where $\pi$ and $\rho$ are epimorphisms. Since $S$ is simple, $\psi$ must be an isomorphism. Hence so is~$\varphi$.
 
(3) Since the left square is a pull-back, there is an exact sequence:
$$
\xymatrix{0\ar[r]&\rad P(S)\ar[r]^{\left(\begin{smallmatrix} i\\u\end{smallmatrix}\right)}&P(S)\oplus P_1\ar[r]&P_0}
$$
Since $\mathcal C$ is hereditary the morphism $\left(\begin{smallmatrix} i\\u\end{smallmatrix}\right)$ is a split monomorphism. Hence there is a morphism $(s,t): P(S)\oplus P_1\rightarrow \rad P(S)$ such that $si+tu=1_{\rad P(S)}$. In particular $s=0$ and we have $u$ is a split monomorphism. \end{proof}

\vskip5pt

\noindent {\bf Proof of Proposition \ref{all strong} (3)}  

Assume $\mathcal D$ is non-strong. By Corollary \ref{non-strong chain} there is an infinite sequence of morphisms:
 $$
\cdots P_3\stackrel{\alpha_2}\rightarrow P_2\stackrel{\alpha_1}\rightarrow P_1\rightarrow Y
$$
with $P_i$'s indecomposable projective objects in $\mathcal D$ and $\alpha_i$'s are monomorphisms. By Proposition \ref{local ar}, each $P_i$ is a projective cover of a simple object $S_i$ which is a subobject of $\cok f$. Since $\cok f$ has a projective cover, we can assume 
$$
0\rightarrow K\rightarrow P\rightarrow \cok f\rightarrow 0
$$
is the minimal projective resolution of $\cok f$. 
%
By Lemma \ref{all strong} (3), every $\rad P_i$ is a summand of $K$. But $\cdots\subsetneq P_{i+1}\subsetneq P_i\subsetneq P_{i-1}\subsetneq \cdots$ implies that  $\cdots\subsetneq \rad P_{i+1}\subsetneq \rad P_i\subsetneq \rad P_{i-1}\subsetneq \cdots$. Therefore, $K$ contains infinitely many indecomposable direct summands, which is a contradiction. $\qed$

\begin{cor} \label{all strong cor}
Let $\mathcal C$ be a $\Hom$-finite, hereditary abelian category with enough projective objects. Then each minimal right determiner is strong. 
\end{cor}


\subsection{Sufficient conditions for the existence of minimal right determiners} The converses of Proposition \ref{ker} and Proposition \ref{cok} are also true. Suppose $f:X\rightarrow Y$ is a right minimal non-isomorphism in a $\Hom$-finite hereditary abelian category $\mathcal C$.  First we have the following observation:

\begin{lem}
For any indecomposable direct summand $L$ of $\Ker f$, assume there exists an almost split sequence $$0\rightarrow L\stackrel{s}\rightarrow M\stackrel{t}\rightarrow N\rightarrow 0.$$ Then $N$ can almost factor through $f$.
\end{lem}

\begin{proof} Denote $u: L\rightarrow \Ker f$ to be the split monomorphism and $i:\Ker f\rightarrow X$ to be the canonical embedding. Then, because $f$ is right minimal, the composition $iu$ is not a split monomorphism. Also because $s$ is left almost split, $iu$ factors through $s$. Hence we have the following commutative diagram:
$$
\xymatrix{0\ar[r]&L\ar[r]^s\ar[d]^u&M\ar[r]^t\ar[d]^v&N\ar[r]\ar[d]^w&0\\
          0\ar[r]&\Ker f\ar[r]^i&X\ar[r]^f&Y}
$$

By Lemma \ref{pullback}, $w$ cannot factor through $f$. Also because $t$ is right almost split, it is easy to see that $w$ can almost factor through $f$. \end{proof}

\begin{lem}
For each simple subobject $S$ of $\cok f$, if $S$ has a projective cover $P(S)$, then $P(S)$ can almost factor through $f$.
\end{lem}

\begin{proof}  By the assumption, we know that there is an exact sequence $$\xymatrix{0\ar[r]&\rad P(S)\ar[r]^\iota&P(S)\ar[r]^\pi&S\ar[r]&0.}$$ Since $P(S)$ is projective, there is a morphism $\alpha: P(S)\rightarrow Y$ such that the right square of the following diagram commutes:

$$
\xymatrix{0\ar[r]&\rad P(S)\ar@{-->}[d]\ar[r]^\iota&P(S)\ar[r]^\pi\ar@{-->}[d]^\alpha&S\ar[r]\ar@{^(->}[d]^i&0\\
           &X\ar[r]^f&Y\ar[r]^p&\cok f\ar[r]&0.}
$$
which implies that $\alpha\iota$ factors through $\Ima f$. Because $\mathcal C$ is hereditary, $\rad P(S)$ is projective. Therefore $\alpha\iota$ factors through $f$. Then it is clear that $P(S)$ can almost factor through $f$. \end{proof}

\vskip10 pt

\begin{prop}\label{suff}
Let $f$ be a right minimal non-isomorphism in $\mathcal C$. If for each indecomposable summand $L_i$ of $\Ker(f)$ there is an almost split sequence $0\rightarrow L_i\rightarrow M_i\rightarrow N_i\rightarrow 0$ and if $\soc\cok(f)=\bigoplus S_i$ is essential where each $S_i$ has a projective cover, then $f$ has a strong minimal right determiner $\mathcal D=add(\{N_i\}\cup \{P(S_i)\})$.
\end{prop}

\begin{proof} From the discussion above we know that $N_i$ and $P(S_i)$ can almost factor through $f$. 

 Let $g: X'\rightarrow Y$ be a morphism which cannot factor through $f$, we need to show that either there is a non-projective object  $N\in \mathcal D$ and a morphism $\phi:N_i\rightarrow X'$  such that $g\phi$ can almost factor through $f$  or there is a projective object $P(S_i)\in \mathcal D$ and $\psi:P(S_i)\rightarrow X'$, such that $g\psi$ can almost factor through $f$. Consider the canonical composition $f: X\rightarrow \Ima f\stackrel{i}\hookrightarrow Y$, it separates into the following two cases:

{\bf Case 1: $g$ can factor through $\Ima f$. } Assume there is $g^\prime: X'\rightarrow \Ima f$ such that $g=ig^\prime$.  Pull back  using $g^\prime$, we have the following commutative diagram:
$$
\xymatrix{\delta:&0\ar[r]&\Ker f\ar@{=}[d]\ar[r]&E\ar[r]\ar[d]&X'\ar[r]\ar[d]^{g^\prime}&0\\
&0\ar[r]&\Ker f\ar[r]&X\ar[r]&\Ima f\ar[r]&0}
$$

Since $g$ cannot factor through $f$, $g^\prime$ cannot factor through $X$. By Lemma \ref{pullback}, $\delta$ does not split. In particular, there is a summand $K$ of $\Ker f$ such that the composition $K\hookrightarrow\Ker f\hookrightarrow E$ is not a split monomorphism. Suppose that $0\rightarrow K\rightarrow M\rightarrow N\rightarrow 0$ is an almost split sequence. Then we can obtain the following commutative diagram:

$$
\xymatrix{
0\ar[r]&K\ar[d]^u\ar[r]&M\ar[r]\ar[d]&N\ar[r]\ar[d]^{\phi}&0\\
0\ar[r]&\Ker f\ar@{=}[d]\ar[r]&E\ar[r]\ar[d]&X'\ar[r]\ar[d]^{g^\prime}&0\\
0\ar[r]&\Ker f\ar[r]&X\ar[r]&\Ima f\ar[r]&0}
$$

Using Lemma \ref{pullback}, $g'\phi$ cannot factor through $X$. Hence it is easy to see that $g\phi$ can almost factor through $f$.

{\bf Case 2: $g$ cannot factor through $\Ima f$.} This is equivalent to  $\pi g\neq0$ which implies that $\Ima\pi g$ is a nonzero subobject of $\cok f$. Since $\soc\cok f$ is essential, there is a simple subobject $S$ such that $S\subseteq\Ima\pi g\subseteq\cok f$. Hence there is a $\psi: P(S)\rightarrow X'$ such that $\pi g\psi=\iota\rho $. 

$$
\xymatrix{0\ar[r]&\rad P(S)\ar[dd]\ar[r]&P(S)\ar[r]^\rho\ar[d]^\psi &S\ar[r]\ar@{^(->}[dd]^\iota&0\\
&&X'\ar[d]^g\\
0\ar[r]&\Ima f\ar[r]^i&Y\ar[r]^\pi&\cok f\ar[r]&0. }
$$

It is easy to see that $g\psi$ can almost factor through $f$.

\end{proof}

Combining the results from Theorem \ref{exist}, Proposition \ref{ker}, \ref{cok}, \ref{strong essential}, \ref{suff}, we have the following main result of this section:

\begin{thm}\label{main}
Suppose $\mathcal C$ is a $\Hom$-finite hereditary abelian $k$-category. Let $f$ be a right minimal non-isomorphism in $\mathcal C$. Then the following are equivalent\\
$(1)$ $f$ has a strong minimal right determiner.\\
$(2)$ $\soc \cok(-,f)$ is essential. \\
$(3)$ Each indecomposable summand $K$ of $\Ker(f)$ is the starting term of an almost split seqeunce
$$
\xymatrix{0\ar[r]&K\ar[r]&M\ar[r]&N\ar[r]&0,}
$$
$\soc\cok(f)$ is essential and each simple subobject of $\cok f$ has a projective cover.\\
Furthermore, if the strong minimal right determiner exists, it is given by $$\mathcal D=add(\{N\}\cup\{P(S)\})$$ where $N$ corresponds to each indecomposable summand of $\Ker f$ and $P(S)$ is the projective cover of $S$ which is a simple subobject of $\cok f$.
\end{thm}

For arbitrary morphism $f: X\rightarrow Y$, the {\bf intrinsic kernel} of $f$, denoted by $\widetilde\Ker f$ is defined to be  $f$ to be $\Ker f_0$ where $f_0: X_0\rightarrow Y$ is the right minimalization of $f$. From Theorem \ref{main} and Proposition \ref{r min ver}, we can recover Theorem \ref{C}.


Finally, together with Corollary \ref{all strong cor} we state the equivalent condition of existence of minimal right determiner in the case that $\mathcal C$ has enough projective objects:
\begin{cor}\label{enough proj} Let $\mathcal C$ be a $\Hom$-finite hereditary abelian $k$-category and $f$ a morphism in $\mathcal C$. 
Assume $\mathcal C$ has enough projective objects. Then $f$ has a minimal right determiner if and only if   \begin{enumerate}
\item each indecomposable summand of the intrinsic kernel  $\widetilde\Ker(f)$ is the starting term of an almost split seqeunce, 
\item $\soc\cok(f)$ is essential. 
\end{enumerate}
\end{cor}

\begin{rmk} \label{enough inj}
This theorem still holds as long as every minimal right determiner is strong.
\end{rmk}

\section{Strongly locally finite quivers}\label{inf q}
In this section, we give a brief introduction to the representation theory of infinite quivers. A detailed survey into this topic was made in \cite{BLP}. And we apply our results in Section 5 to study the minimal right determiners of morphisms in the category of representations of infinite quivers.

\subsection{Infinite quivers}  $Q=(Q_0,Q_1)$ stands for a quiver, where $Q_0$ stands for the set of vertices and $Q_1$ for the set of arrows.  A quiver $Q$ is called {\bf locally finite} if for each vertex $x$, there is a finite number of arrows staring from and ending in $x$. $Q$ is called {\bf interval finite} if for $\forall x,y\in Q_0$, the number of paths from $x$ to $y$ is finite. It is easy to see that $Q$ has no oriented cycle if $Q$ is interval finite. 

\begin{defn}\label{slf}
$Q$ is called {\bf strongly locally finite} if it is both locally finite and interval finite.
\end{defn}
A {\bf left infinite path} in $Q$ is a path:
$$
\xymatrix{\cdots\ar[r]&\circ\ar[r]&\circ\ar[r]&\circ}
$$

A {\bf right infinite path} in $Q$ is a path:
$$
\xymatrix{\circ\ar[r]&\circ\ar[r]&\circ\ar[r]&\cdots}
$$

A {\bf double infinite path} in $Q$ is a path:
$$
\xymatrix{\cdots\ar[r]&\circ\ar[r]&\circ\ar[r]&\circ\ar[r]&\cdots}
$$

\subsection{Representations}\label{reps} Let $k$ be an arbitrary field and $Q$ is a strongly locally finite quiver. A representation $M$ of $Q$ over $k$ is a family of $k$-spaces $M(x) $ with $x \in Q_0$ together with a family of $k$-maps $M(\alpha): M(x)\rightarrow M(y)$ with $\alpha: x\rightarrow y$ in $Q_1$.  A morphism $f : M \rightarrow N$ of $k$-representations of $Q$ consists of a family of $k$-maps $f(x) : M(x) \rightarrow N(x)$ with $x \in Q_0$ such that $f(y)M(\alpha) = N(\alpha)f(x)$, for every arrow $\alpha : x \rightarrow y$. Denote by $Rep(Q)$ the abelian category of all $k$-representations of $Q$.

A $k$-representation $M$ is called {\bf locally finite dimensional} if $M(x)$ is of finite $k$-dimension for all $x \in Q_0$; and {\bf finite dimensional} if $\sum_{x\in Q_0} \dim M(x)$ is finite. Let $rep(Q)$ and $rep^b(Q)$ denote the full subcategories of $Rep(Q)$ generated by locally finite dimensional representations and finite dimensional representations respectively. Notice that $rep(Q)$ is a hereditary abelian category, but not necessarily $\Hom$-finite.

For each vertex $x\in Q_0$, let $P_x$ (or $I_x$) be the indecomposable projective (or injective) representation at $x$. Denote by $proj(Q)$ (or $inj(Q)$) the additive subcategory of $rep(Q)$ of $P_x$ (or $I_x$), $x\in Q_0$.

\begin{defn}\label{rep(q)}
A representation $M\in rep(Q)$ is called {\bf finitely presented} if there are $P_0$, $P_1\in proj(Q)$ such that there is a projective resolution: 
$$
\xymatrix{0\ar[r]&P_1\ar[r]&P_0\ar[r]&M\ar[r]&0.}
$$

A representation $M\in rep(Q)$ is called {\bf finitely co-presented} if there are $I_0$, $I_1\in inj(Q)$ such that there is a injective resolution: 
$$
\xymatrix{0\ar[r]&M\ar[r]&I_0\ar[r]&I_1\ar[r]&0}.
$$

Denote by $rep^+(Q)$ (or $rep^-(Q)$) the category of finitely (co-)presented representations.
\end{defn}

\begin{rmk}
It is important to remind the readers that in some references $P_x$ (or $I_x$) are called the principle projective (or injective) representations. They are not necessarily all the projective (or injective) objects in $rep(Q)$. For example, the injective representation $I_1$ of the following quiver is a projective object in $rep(Q)$ but not of the form $P_x$.
$$
\xymatrix{\cdots\ar[r]&3\ar[r]&2\ar[r]&1}
$$

\end{rmk}

\begin{prop}[\cite{BLP} 1.15]\ 
\begin{enumerate}
\item $rep^+(Q)$ and $rep^-(Q)$ are $\Hom$-finite, hereditary and abelian, which are extension closed in $rep(Q)$.
\item $rep^+(Q)\cap rep^-(Q)=rep^b(Q)$. 
\end{enumerate}
\end{prop}

As a result of Proposition \ref{all strong} we have 

\begin{cor}\label{all all strong}
Let $\mathcal C$ be $rep^+(Q)$ or $rep^-(Q)$ for some strongly locally finite quiver $Q$.  Then each minimal right determiner in $\mathcal C$ is strong.
\end{cor}

\subsection{Almost split sequences} Let $Q$ be a strongly locally finite quiver.  The opposite quiver $Q^{op}$ is a quiver obtained by reversing all the arrows in $Q$. The functor $D=\Hom_k(-,k)$ defines a duality $D:rep(Q)\rightarrow rep(Q^{op})$ satisfying $D(P_x)=I_{x^o}$ and $D(I_x)=P_{x^o}$, where $x^o$ denotes the same vertex $x$ in the opposite quiver.

While take $A=\bigoplus_{x\in Q_0}P_x$ as a representation. The functor $\Hom_A(-,A)$ defines a duality $\Hom_A(-,A):proj(Q)\rightarrow proj(Q^{op})$ satisfying $\Hom_A(P_x, A)\simeq P_{x^o}$.

\begin{prop}[\cite{BLP} 1.19]
There is a Nakayama equivalence:
$$
\nu=D\Hom_A(-,A): proj(Q)\rightarrow inj(Q): P_x\rightarrow I_x
$$
\end{prop}

Using the Nakayama functor, we can define the Auslander-Reiten translation as follows. If $M\in rep^+(Q)$ with minimal projective resolution
$$
\xymatrix{0\ar[r]&P_1\ar[r]^f&P_0\ar[r]&M\ar[r]&0,}
$$
define $TrM\in rep^+(Q^{op})$ to be $\cok\Hom(f,A)$ and define $DTrM\in rep^-(Q)$ to be $\Ker \nu(f)$:
 $$
\xymatrix{0\ar[r]&DTrM\ar[r]&\nu(P_1)\ar[r]^{\nu(f)}&\nu(P_0)\ar[r]&0.}
$$
Dually we can define $TrDM$ for $M\in rep^-(Q)$.

Definitions of almost split sequences in the category $rep(Q)$, $rep^+(Q)$ and $rep^-(Q)$ are the same as in section \ref{ars def}.
We have the existence theorem for almost split sequences:

\begin{thm}[\cite{BLP} 2.8]
Let $Q$ be a strongly locally finite quiver, and let $M\in rep(Q)$ be indecomposable.\begin{enumerate}
\item If $M\in rep^+(Q)$ is not projective, then $rep(Q)$ admits an almost split sequence \\$\xymatrix{0\ar[r]&DTrM\ar[r]&N\ar[r]&M\ar[r]&0}$, where $DTrM\in rep^-(Q)$.
\item If $M\in rep^-(Q)$ is not injective, then $rep(Q)$ admits an almost split sequence \\$\xymatrix{0\ar[r]&M\ar[r]&N\ar[r]&TrDM\ar[r]&0}$, where $TrDM\in rep^+(Q)$.
\end{enumerate}
Furthermore, these are all the almost split sequences in $rep(Q)$.
\end{thm}

\begin{prop}[\cite{BLP} 3.6]\label{k findim}
If $\xymatrix{0\ar[r]&L\ar[r]&M\ar[r]&N\ar[r]&0}$ is a short exact sequence in $rep(Q)$, then it is an almost split sequence in $rep^+(Q)$ if and only if it is an almost split sequence in $rep(Q)$ with $L\in rep^b(Q)$.
\end{prop}

Notice that not every indecomposable object in $rep^+(Q)$ admits an almost split sequence. We have the following definition:

\begin{defn}\label{lr ar cat}
We say that an additive category $\mathcal C$ is a {\bf right Auslander-Reiten category} if every indecomposable object in $\mathcal C$ is the ending term of a minimal right almost split monomorphism or the ending term of an almost split sequence; a {\bf left Auslander-Reiten category} if every indecomposable object in $\mathcal C$  is the staring term of a minimal left almost split epimorphism or the starting term of an almost split sequence; and an {\bf Auslander-Reiten category} if it is left and right Auslander-Reiten.
\end{defn}

When $\mathcal C$ is an abelian category, $\mathcal C$ is right Auslander-Reiten if and only if every indecomposable non-projective object is the ending term of an almost split sequence and every indecomposable projective object has a simple top; $\mathcal C$ is left Auslander-Reiten if and only if every indecomposable non-injective object is the starting term of an almost split sequence and every indecomposable injective object has a simple socle (\cite{BLP}, 2.2). 
\vskip5 pt

Due to \cite{BLP}, we have the following characterization of right Auslander-Reiten category depending on the combinatorial property of the quiver:

\begin{thm}[\cite{BLP}, 3.7]\label{left ar}
If $Q$ is a strongly locally finite quiver, then\begin{enumerate}
\item $rep^+(Q)$ is left Auslander-Reiten if and only if $Q$ has no right infinite path; 
\item $rep^+(Q)$ is right Auslander-Reiten if and only if $Q$ has no left infinite path, or else $Q$ is a left infinite or double infinite path.
\end{enumerate}
\end{thm}

Because $D: rep^+(Q)\rightarrow rep^-(Q^{op})$ is a duality, there is  a dual version of this theroem: 
\begin{cor}
If $Q$ is a strongly locally finite quiver, then
\begin{enumerate}
\item $rep^-(Q)$ is right Auslander-Reiten if and only if $Q$ has no left infinite path; 
\item $rep^-(Q)$ is left Auslander-Reiten if and only if $Q$ has no right infinite path, or else $Q$ is a right infinite or double infinite path.
\end{enumerate}
\end{cor}


Combining with Corollary \ref{all strong cor}, we have:

\begin{cor} \label{mrd quiver}
Let $Q$ be a strongly locally finite quiver and $f:X\rightarrow Y$ is a morphism in the category $rep^+(Q)$. Then $f$ has a unique minimal right determiner if and only if the intrinsic kernel $\widetilde\Ker f$ is in $rep^b(Q)$, $\soc\cok f$ is essential. In this case, the minimal right determiner is given by the formula $TrD \widetilde\Ker f\oplus P(\soc\cok f)$.
\end{cor}

 We are going to prove Theorem \ref{no r inf path}, which gives a characterization of the existence of minimal right determiner depending on the combinatorial property of the quiver:
 \\
 
\noindent {\bf Theorem D.}
{\it
Let $Q$ be a strongly locally finite quiver. Then the following are equivalent:\\
\indent $(1)$ $Q$ has no right infinite paths.\\
\indent $(2)$ $rep^+(Q)$ is left Auslander-Reiten.\\
\indent $(3)$ Any morphism $f:X\rightarrow Y$ in $rep^+(Q)$ has a unique minimal right determiner.
} 

\begin{proof}
(1) $\iff$ (2) follows directly from Theorem \ref{left ar} (1). It suffice to prove (1) $\iff$ (3)

(1) $\Rightarrow$ (3): Suppose $Q$ has no right infinite path. Then $rep^+(Q)$ is left Auslander-Reiten. So each indecomposable non-injective object is the starting term of an almost split sequence. By Proposition \ref{k findim}, each indecomposable non-injective object is in $rep^b(Q)$. Hence, for any right minimal non-isomorphism $f$ in $rep^+(Q)$, $\Ker f$ is in $rep^b(Q)$. On the other hand, since $Q$ has no right infinity path, $\soc M$ is essential for any $M\in rep^+(Q)$ (\cite{BLP}, 1.1). Hence, for any morphism $f$ in $rep^+(Q)$, $\soc\cok f$ is essential.
Therefore, by Corollary \ref{mrd quiver}, any morphism $f$ has a minimal right determiner.

(3) $\Rightarrow$ (1): Suppose any morphism $f$ in $rep^+(Q)$ has a minimal right determiner. Then by Corollary \ref{mrd quiver}, the intrinsic kernel of $f$ is in $rep^b(Q)$. For any indecomposable non-injective object $M\in rep^+(Q)$, since $M$ is non-injective, then there is a non-split monomorphism $f:M\rightarrow N$. Therefore $M$ is the intrinsic kernel of the canonical epimorphism $N\rightarrow \cok f$. Hence $M$ is in $rep^b(Q)$. If $Q$ has a right infinite path $p$ with initial arrow $x\rightarrow y$,  then $P_y\not\in rep^b(Q)$ because $p$ is infinite. Therefore, $P_y$ must be an injective object in $rep^+(Q)$. However, this implies that the embedding $P_y\hookrightarrow P_x$ splits, which is absurd. 
\end{proof}
\vskip5pt

As mentioned at the beginning there is a dual notion of left determined morphisms and also minimal left determiners. It is left for the readers to formulate the details. With that, we can give the following dual statement:

\begin{cor} 
Let $Q$ be a strongly locally finite quiver. Then the following are equivalent:\\
$(1)$ $Q$ has no left infinite path.\\
$(2)$ $rep^-(Q)$ is right Auslander-Reiten.\\
$(3)$ Every morphism $f:X\rightarrow Y$ in the category $rep^-(Q)$ has a unique minimal left determiner.\\
\end{cor} 

\section{Examples}
In this section, we provide examples of right determiners in various categories. In some of them, the existence or uniqueness or the formula of the minimal right determiner might fail. 

\begin{exm} This is a classical example where $\mathcal C=\Lambda\textendash\modd$ for some hereditary artin algebra $\Lambda$. We will show that the minimal right determiner of a morphism agrees with Auslander-Reiten-Smal{\o}-Ringel formula.  \end{exm}
Take the quiver $Q=\xymatrix{1&2\ar[l]&3\ar[l]}$. Consider the non-zero morphism of $kQ$-modules $f:P_2\rightarrow I_2$.
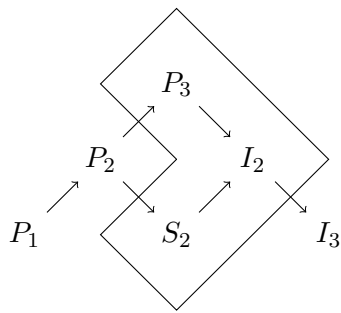
\begin{center}
\begin{tikzpicture}[->]
\path (0,0) node (p1) {$P_1$}
(1,1) node (p2) {$P_2$}
(2,2) node (p3) {$P_3$}
(2,0) node (s2) {$S_2$}
(3,1) node (i2) {$I_2$}
(4,0) node (i3) {$I_3$};

\draw(p1)--(p2);
\draw(p2)--(p3);
\draw(p2)--(s2);

\draw(s2)--(i2);
\draw(p3)--(i2);
\draw(i2)--(i3);

\draw[-](2,1)--(1,0)--(2,-1)--(4,1)--(2,3)--(1,2)--(2,1);
\end{tikzpicture}
\end{center}

The support of the cokernel functor $\cok(-,f)$ is indicated by the modules in the box. As we can see, the minimal right determiner of $f$ is $S_2\oplus P_3\simeq \tau^-\Ker f\oplus P(\soc\cok f)$ which is exactly given by the Auslander-Reiten-Smal{\o}-Ringel formula.


\begin{exm} In this example, we show that the condition $\mathcal C$ being hereditary is necessary in the Auslander-Reiten-Smal{\o}-Ringel formula.  \end{exm}
Consider the Nakayama algebra $kQ/<\gamma\beta\alpha, \beta\alpha\gamma>$ given by:
$$
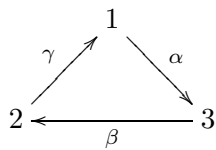
\xymatrix{&1\ar[dr]^\alpha&\\
2\ar[ur]^\gamma&&3\ar[ll]^\beta}
$$

 and the morphism $f:\begin{smallmatrix}3\\2\\1\end{smallmatrix}\longrightarrow \begin{smallmatrix}1\\3\\2\end{smallmatrix}$. The AR-quiver is given below:
\vskip10pt

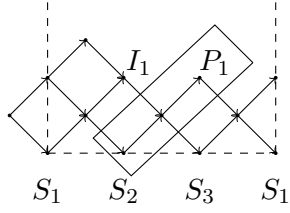
\begin{center}
\begin{tikzpicture}[->]
\foreach \x in {0,1,2,3}
\foreach \y in {0,1}
{
\filldraw (\x,\y) circle(0.5pt);
}
\foreach \x in {-0.5,0.5,1.5,2.5}
\foreach \y in {0.5}
{
\filldraw (\x,\y) circle(0.5pt);
}
\filldraw(.5,1.5) circle(0.5pt);

\foreach \x in {0,1,2}
\foreach \y in {0}
{
\draw (\x,\y)--(\x+0.5,\y+0.5);
}
\foreach \x in {-0.5,0.5,1.5,2.5}
\foreach \y in {0.5}
{
\draw (\x,\y)--(\x+0.5,\y+0.5);
}
\draw(0,1)--(0.5,1.5);

\foreach \x in {-0.5,0.5,1.5,2.5}
\foreach \y in {0.5}
{
\draw (\x,\y)--(\x+0.5,\y-0.5);
}
\foreach \x in {0,1,2}
\foreach \y in {1}
{
\draw (\x,\y)--(\x+0.5,\y-0.5);
}
\draw(0.5,1.5)--(1,1);

\draw[-][dashed](0,0)--(0,2);
\draw[-][dashed](3,0)--(3,2);

\draw[-] (0,-0.5) node {$S_1$};
\draw[-] (1,-0.5) node {$S_2$};
\draw[-] (2,-0.5) node {$S_3$};
\draw[-] (3,-0.5) node {$S_1$};

\draw[-] (1.2,1.2) node {$I_1$};
\draw[-] (2.2,1.2) node {$P_1$};
\draw[-] (2.7,1.2)--(1.1,-.3)--(.6,0.2)--(2.2,1.7)--(2.7,1.2);
\draw[-][dashed](0,0)--(1,0)--(2,0)--(3,0);
\end{tikzpicture}
\end{center}

The minimal right determiner of $f$ is just $S_2\simeq \tau^-\Ker f$.


\begin{exm}\label{n_ex} We construct morphisms that do not have minimal right determiners to illustrate that each one of the conditions in Theorem \ref{main} $(3)$ is necessary for the existence of a minimal right determiner. Let $Q$ be the following strongly locally finite quiver:
$$\xymatrix{\mathop{\cdot}\limits^1\ar[r]&\mathop{\cdot}\limits^2\ar[r]&\mathop{\cdot}\limits^3\ar[r]&\cdots}$$
\end{exm}

(1) Let $rep^+(Q)$ be the category of finitely presented representations. The AR-quiver contains two connected components:

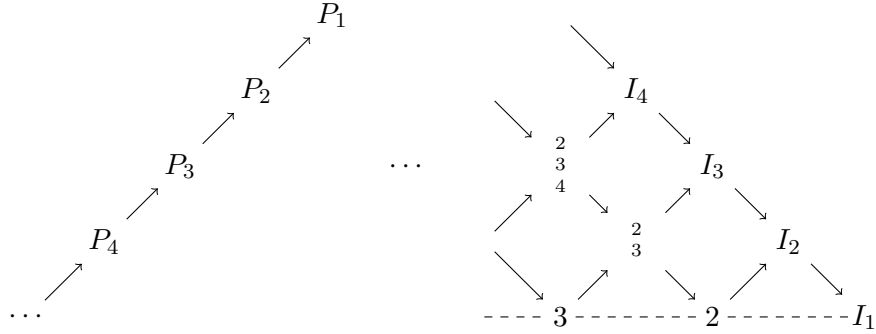
\begin{center}
\begin{tikzpicture}[->]
\path (3,3) node (p1) {$P_1$}
( 2,2) node (p2) {$P_2$}
( 1,1) node (p3) {$P_3$}
( 0,0) node (p4) {$P_4$}
(-1,-1) node (p5) {$\cdots$};

\draw (p5)--(p4);
\draw (p4)--(p3);
\draw (p3)--(p2);
\draw (p2)--(p1);

\path (10,-1) node (1) {$I_1$}
(9,0) node (2) {$I_2$}
(8,1) node (3) {$I_3$}
(7,2) node (4) {$I_4$}

(8,-1) node (11) {$2$}
(7,0) node (21) {$\tiny\begin{array}{c}
  2  \\
  3
\end{array}$}
(6,-1) node (13) {$3$}
(6,1) node(31) {\tiny{{$\begin{array}{c}
  2  \\
  3  \\
  4
\end{array}$}}}

(5,2) node (void1) {}
(5,0) node (void2) {}
(6,3) node (void3) {}
(4,1) node (a) {$\cdots$};

\draw (2)--(1);
\draw (3)--(2);
\draw (4)--(3);

\draw (31)--(4);
\draw (21)--(3);
\draw (11)--(2);

\draw (31)--(21);
\draw (21)--(11);
\draw (13)--(21);

\draw (void1)--(31);
\draw (void2)--(31);
\draw (void2)--(13);
\draw (void3)--(4);

\draw[-][dashed] (5,-1)--(5.8,-1);
\draw[-][dashed] (6.2,-1)--(7.8,-1);
\draw[-][dashed](8.2,-1)--(9.8,-1);
\end{tikzpicture}
\end{center}

\noindent Consider the non-zero morphism $f: P_1\rightarrow I_1$.
Then $f$ is right determined by the set of all finitely generated injective representations $\{I_i\}$. However, the minimal right determiner does not exist, because $\Ker f\simeq P_2$ is not the starting term of an almost split sequence.

\noindent Consider the morphism $g: 0\rightarrow P_1$. It is right determined by the set of all finitely generated projective representations $\{P_i\}$. However, the minimal right determiner does not exist,  because $\soc\cok g=\soc P_1=0$ which is not essential.

(2) $rep^-(Q)$ is the category of finitely copresented modules. The AR-quiver of $rep^-(Q)$ contains just the AR-component which contains injectives. 

\noindent Consider the morphism $h: 0\rightarrow I_1$. 
It is right determined by the set of all finitely generated injective representations $\{I_i\}$. However, the  minimal right determiner does not exist, because $I_1\simeq\soc\cok h$ does not have a projective cover.


\begin{exm} We give an example where the minimal right determiner of a morphism exists and it consists of infinitely many non-isomorphic indecomposable objects. Let $Q$ be a strongly locally finite quiver, and consider again the category of finitely presented representations $rep^+(Q)$: \end{exm}
\begin{center}
\begin{tikzpicture}[->]
\path
(0,0) node (0) {$0$}
(1,0) node (2) {$2$}
(2,0) node (4) {$4$}
(0,1) node (1) {$1$}
(1,1) node (3) {$3$}
(2,1) node (5) {$5$}
(3,0.5) node (l) {$\cdots$};

\draw (1)--(0);
\draw (1)--(3);
\draw (3)--(2);
\draw (3)--(5);
\draw (5)--(4);

\draw(5)--(2.5,1);
\end{tikzpicture}
\end{center}

\noindent Consider the non-zero morphism $f: S_0\rightarrow P_1$ in $rep^+(Q)$. Then $f$ has a minimal right determiner $\mathcal D=\add\{P_{2i} | i>0\}$, since each $S_{2i}$ is a subobject of $\cok f$, for $i>0$. (see Theorem \ref{main})


\begin{exm} Sometimes the minimal right determiner formula (Theorem \ref{main}) still holds even when the $\Hom$-finite condition does not hold, which we illustrate in the following example: \end{exm}
\noindent Consider the following quiver $Q$ which is an  infinite quiver with no infinite paths:
$$
\xymatrix{\cdots&\mathop{\cdot}\limits^{-3}\ar[l]\ar[r]&\mathop{\cdot}\limits^{-2}
&\mathop{\cdot}\limits^{-1}\ar[l]\ar[r]&\mathop{\cdot}\limits^{0}
&\mathop{\cdot}\limits^{1}\ar[l]\ar[r]&\mathop{\cdot}\limits^{2}
&\mathop{\cdot}\limits^{3}\ar[l]\ar[r]&\cdots}
$$
Take $\mathcal C$ to be the category of locally finite dimensional representations $rep(Q)$ which is an abelian AR category. However, since the object $\mathop{\bigoplus}\limits_{i=0}^{\infty}S_i\in \mathcal C$ is an infinite direct sum of indecomposable objects, $\mathcal C$ is not Krull-Schmidt (hence not $\Hom$-finite).
Let $M$ be the indecomposable representation $M(x)=k$ for all $x\in Q_0$ and consider the embedding $f:S_0\rightarrow M$. Then $f$ has a minimal right determiner $\mathcal D=\add\{P_{2i}| i\neq 0\}$.

\begin{exm}  We show an example of a non-Krull-Schmidt hereditary category and compute the minimal right determiners of some morphisms. 
We also show that there is a morphism such that the minimal right determiner consists of indecomposable objects which cannot almost factor through that morphism. \end{exm}

Let $\mathbb Z\textendash\modd$ denote the category of all finitely generated $\mathbb Z$ modules. As we all know, every finitely generated $\mathbb Z$ module is a direct sum of a torsion free module and torsion module.

Suppose $p$ is a prime number, then all the almost split sequence for torsion modules are $0\rightarrow\mathbb Z_p\rightarrow \mathbb Z_{p^2}\rightarrow \mathbb Z_p\rightarrow0$ and $0\rightarrow\mathbb Z_{p^n}\rightarrow \mathbb Z_{p^{n+1}}\oplus Z_{p^{n-1}}\rightarrow \mathbb Z_{p^n}\rightarrow0$ for $n>1$.

It is also known that the $\mathbb Z$-homomorphism $f:\mathbb Z\rightarrow \mathbb Z$ is irreducible if and only if $f$ is multiplication by $p$ for some prime number $p$. The $\mathbb Z$-homomorphism $f:\mathbb Z\rightarrow \mathbb Z_{p^n}$ is never irreducible.

In conclusion, the AR-quiver of $\mathbb Z\textendash\modd$ consists of the following connected components for each prime number $p$:
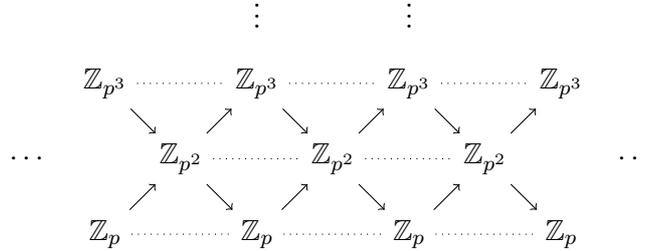
$$
\begin{tikzpicture}[->]
\path (-3,0) node (2-2) {$\mathbb Z_{p}$}
(-1,0) node (20) {$\mathbb Z_{p}$}
(1,0) node (22) {$\mathbb Z_{p}$}
(3,0) node (24) {$\mathbb Z_{p}$}

(-2,1) node (4-2) {$\mathbb Z_{p^2}$}
(0,1) node (40) {$\mathbb Z_{p^2}$}
(2,1) node (42) {$\mathbb Z_{p^2}$}

(-3,2) node (31) {$\mathbb Z_{p^3}$}
(-1,2) node (32) {$\mathbb Z_{p^3}$}
(1,2) node (33) {$\mathbb Z_{p^3}$}
(3,2) node (34) {$\mathbb Z_{p^3}$}

(-4,1) node {$\cdots$}
(4,1) node {$\cdots$}
(-1,3) node {$\vdots$}
(1,3) node {$\vdots$};

\draw (31)--(4-2);
\draw (2-2)--(4-2);
\draw (4-2)--(32);
\draw (4-2)--(20);

\draw (32)--(40);
\draw (20)--(40);
\draw (40)--(33);
\draw (40)--(22);

\draw (33)--(42);
\draw (22)--(42);
\draw (42)--(34);
\draw (42)--(24);

\draw[dotted][-] (2-2)--(20);
\draw[dotted][-] (20)--(22);
\draw[dotted][-] (22)--(24);

\draw[dotted][-] (4-2)--(40);
\draw[dotted][-] (40)--(42);

\draw[dotted][-] (31)--(32);
\draw[dotted][-] (32)--(33);
\draw[dotted][-] (33)--(34);
\end{tikzpicture}
$$
together with the connected component which contains only the indecomposable module $\mathbb Z$ and irreducible morphisms given by multiplications by $p$ for all prime numbers $p$:

%
%
%
%

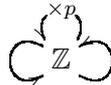
$$
\xymatrix{\mathbb Z\ar@(ul,dl)\ar@(dr,ur)\ar@(ur,ul)[]|{\times p}}
$$

Now consider the morphism $f:\mathbb Z_{p^m}\rightarrow \mathbb Z_{p^n}$ whose image is isomorphic to $\mathbb Z_{p^k}$, $k\leq n$.
If $k<n$, then $f$ has a minimal right determiner $\mathbb Z_{p^{m-k}}\oplus \mathbb Z$. Otherwise, $f$ has a minimal right determiner $\mathbb Z_{p^{m-k}}$.

The $\mathbb Z$-module homomorphism $g: \mathbb Z\rightarrow \mathbb Z_{p^n}$ is right determined by a set of modules $S_N=\{\mathbb Z_{p^k}| k\geq N\}$, for all $N$. But $g$ does not have a minimal right determiner.

Here is an important observation: $h=\times p: \mathbb Z\rightarrow\mathbb Z$ is always right determined by $\mathbb Z$ and $\mathbb Z$ is the unique minimal right determiner. However, it is not semi-strong (see Definition \ref{semi strong def}).


\vskip10pt

\begin{exm} \label{1221} In a non-$\Hom$-finite additive category, we show that the minimal right determiner may not be unique. 
Let $\mathcal C$ to be the path category of the following quiver:

$$
Q:\xymatrix{1\ar@/^/[r]&2\ar@/^/[l]}
$$
\noindent i.e. indecomposable objects in $\mathcal {C}$ are given by vertices $\{v_1, v_2\}$, and the set of morphisms $\Hom(v_i,v_j)$ are $k$-linear spaces generated by finite paths from $v_i$ to $v_j$.

\noindent Consider the zero morphism $f: 0\rightarrow v_1$.  Both $\add\{v_1\}$ and $\add\{v_2\}$ are the minimal right determiners of $f$.  
\end{exm}


\begin{exm}\label{wedge close} We construct a Krull-Schmidt non-$\Hom$-finite additive category such that
\begin{enumerate}
\item There is a morphism $f$ which has a unique minimal right determiner $\mathcal D$.
\item The set of all the right determiners of $f$ is not closed under intersections.
\item $\mathcal D$ is not semi-strong.
\end{enumerate}
 \end{exm}
Let $\mathcal C$ be an additive category consists of non-isomorphic indecomposable objects $X,Y, Z$ and $D_i$, for $i>0$. Let $f: X\rightarrow Y$ and a sequence of morphisms $g_mg_{m+1}\cdots g_n$, $\forall n>m>0$:

$$
\xymatrix{ \cdots\ar[r]&D_3\ar[r]^{g_6}&Z\ar[r]^{g_5}&D_2\ar[r]^{g_4}&Z\ar[r]^{g_3}&D_1\ar[r]^{g_2}&Z\ar[r]^{g_1}&Y}
$$
be all the non-zero morphisms between non-isomorphic indecomposable objects in $\mathcal C$ up to scalar multiplication, such that $g_1g_2\cdots g_n$ cannot factor through $f$, for all $n$.

Assume $Z$ and $D_i$ are pairwise non-isomorphic indecomposable objects. Then $f$ has a unique minimal right determiner $\mathcal D= \add\{Z\}$ and $f$ is also right determined by $\mathcal D_k=\add \{D_i\mid i\geq k \}$, $\forall k\geq1$ which are not minimal. So the set $\mathcal S$ of all the right determiners of $f$ is not closed under taking intersections. Notice that $\mathcal D$ is not semi-strong, since $Z$ cannot almost factor through~$f$.
\vskip10pt
 
 \begin{exm} \label{non-strong ex} We show an example of semi-strong but not strong minimal right determiner. \end{exm} 
 Let $Q$ be a strongly locally finite quiver:
 $$
 \xymatrix { 1\ar[r]\ar[d]&2\ar[r]\ar[d]&3\ar[r]\ar[d]&\cdots\\
                   1'\ar[r]&2'\ar[r]&3'\ar[r]&\cdots }
 $$
 Let $\mathcal C=rep(Q)$ be the category of locally finite representations (see Seciont \ref{reps}) and $M$ be the representation  
 $$
\xymatrix{k\ar[r]^1\ar[d]^1&k\ar[r]^1\ar[d]^1&k\ar[r]^1\ar[d]^1&\cdots\\
                   k&k&k&\cdots}
 $$
 
The morphism $f: 0\rightarrow M\oplus P_{1'}$ has a minimal right determiner $\add\{ P_{i'}, i'\geq 1\}$. But it is not strong, since the morphism $g: P_{1'}\stackrel{\left(\begin{smallmatrix}0\\1\end{smallmatrix}\right)}\rightarrow M\oplus P_{1'}$ is an absolute non-factorization of $f$.  Notice that  $\soc (M\oplus P_{1'})= \bigoplus_{i'\geq 1} S_{i'}$ is not essential.  The support of the functor $\cok(-f)$~is

$$
\tiny\xymatrix{ 
P_{2'}\ar[rd]&P_{3'}\ar[d]&\cdots\ar[ld]\\
P_{1'}\ar[r]&M\oplus P_{1'}&P_{1'}\ar[l]&P_{2'}\ar[l]&\cdots\ar[l]}
$$
 
 We still do not know whether every minimal right determiner in a $\Hom$-finite hereditary abelian category is strong. But this example may provide a negative answer to this question. So we conjecture that the morphism $f$ in this example fits in a $\Hom$-finite hereditary abelian category $\mathcal C'$ which we now define:

 \begin{conj}
 Let $\mathcal C^\prime$ be a full subcategory of $rep(Q)$ from Example \ref{non-strong ex} obtained inductively:
$\mathcal C^\prime_0$ is the full subcategory of $rep(Q)$ consisting of objects which are finite direct sums of $M$ and $P_{i'}$, $\forall i'\geq1$.  Define $\mathcal C^\prime_n$ to be the full subcategory of $rep(Q)$ closed under direct summands and finite direct sums, containing $\Ker f$ and $\cok f$ for any morphism $f\in \mathcal C'_{n-1}$ and also contains objects $B$ which have an exact sequences:
$$
0\rightarrow A\rightarrow B\rightarrow C\rightarrow 0
$$
 for some $A,C\in \mathcal C^\prime_{n-1}$. Then $\mathcal C^\prime:= \bigcup_{n\geq0} \mathcal C^\prime_n$ is a $\Hom$-finite hereditary abelian category.
 \end{conj}
%
 
Last but not least, we are going to show a couple of examples of $\Hom$-finite hereditary abelian $k$-categories which are not obtained from a strongly locally finite quiver.   

\begin{exm} (Continuous $A^\infty$-type) Define a $k$-category $\mathcal C$ as follows:  \end{exm}
The indecomposable objects are intervals $(a,b]$ with $a,b\in\mathbb R$. 
Morphisms are defined by 
\begin{eqnarray*}
\Hom((a,b],(c,d])\simeq 
\begin{cases}
k  & \text{if $a\leq c<b\leq d$ }\\
0  & \text{otherwise}
\end{cases}
\end{eqnarray*}

It is clear that $\emptyset$ is the zero object and $(-\infty, x]$ are the indecomposable projective objects.  Every non-split epimorphism $f$ does not have a minimal right determiner in this category.

The category $\mathcal C$ is also called the {\bf representation of real line} and has been studied in \cite{IT}.

%

\begin{exm}
(Mixed  $A^\infty$-type) Define a $k$-category $\mathcal C$ as follows: 
 \end{exm}
The indecomposable objects are
intervals $(a,b]$ with $a,b\in (-\infty,0]\cup\{1,2,3,\cdots\}$.  
Morphisms are defined by 
\begin{eqnarray*}
\Hom((a,b],(c,d])\simeq 
\begin{cases}
k  & \text{if $a\leq c<b\leq d$ }\\
0  & \text{otherwise}
\end{cases}
\end{eqnarray*}

In the positive half line, $\mathcal C$ behaves similar to the representations of the infinite quiver:
$$
\xymatrix{0&1\ar[l]&2\ar[l]&\cdots\ar[l]}
$$
We have simple objects $(k,k+1]$ for $k\geq0$ and almost split sequences are given by:
$$
\xymatrix{0\ar[r]&(a,a+1]\ar[r]&(a,a+2]\ar[r]&(a+1,a+2]\ar[r]&0}
$$
when $a\geq0$ and
$$
\xymatrix{0\ar[r]&(a,b]\ar[r]&(a,b+1]\oplus (a+1,b]\ar[r]&(a+1,b+1]\ar[r]&0}
$$
when $1\leq a+1\leq b$.

In the negative part $\mathcal C$ behaves similar to the continuous $A^\infty$-type.  Any non-isomorphism $f: (a,b]\rightarrow (c,d]$ does not have a minimal right determiner when $c<b<0$.

A non-zero morphism $f:(a,0]\rightarrow (c,d]$ has a minimal right determiner if and only if $a=c$. In this case the minimal right determiner is $(-\infty, 1]$.

\end{document}